\documentclass[a4paper,12pt,intlimits,oneside]{amsart}
\usepackage{amsmath,amsthm,amssymb}
\usepackage{dsfont}
\usepackage{mathrsfs}
\usepackage[style=alphabetic,maxnames=200,maxalphanames=6,giveninits]{biblatex}
\usepackage[colorlinks=true]{hyperref}
\usepackage[utf8]{inputenc}
\usepackage{esint}
\usepackage{amsthm}
\usepackage{latexsym}
\usepackage{amssymb}
\usepackage{xcolor}
\usepackage{hyperref}
\usepackage{bbold}
\usepackage{enumerate}

\makeatletter
\numberwithin{equation}{section}
\theoremstyle{plain}
        \newtheorem{theorem}{Theorem}[section]
        \newtheorem{proposition}[theorem]{Proposition}
        \newtheorem{lemma}[theorem]{Lemma}
        \newtheorem{corollary}[theorem]{Corollary}

        \newtheorem{remark}[theorem]{Remark}  
         
        \newtheorem*{claim*}{Claim}  
         
        \newtheorem{example}{Example}
\newtheorem*{theorem*}{Theorem}
\newtheorem*{definition*}{Definition}
\newtheorem*{proposition*}{Proposition}
\usepackage{graphicx}
\let\oldmarginpar\marginpar
\renewcommand\marginpar[1]{\-\oldmarginpar[\raggedleft\footnotesize #1]
{\raggedright\footnotesize #1}}
\newcommand \be  {\begin{equation}}
\newcommand \ee  {\end{equation}}
\newcommand \la \langle
\newcommand \ra \rangle

\newcommand \eps {\varepsilon}

\newcommand{\R}{\mathbb{R}}
\newcommand{\C}{\mathbb{C}}
\newcommand{\T}{\mathbb{T}}

\newcommand{\N}{\mathbb{N}}
\newcommand{\cC}{\mathcal{C}}

\newcommand{\cW}{\mathcal{W}}
\newcommand{\lS}{\mathds{S}}

\renewcommand{\d}{\partial}

\title[Turbulent cascades]{Turbulent Cascades for a family of damped Szeg\H{o} equations}
\author[P. Gérard]{Patrick G\'erard}
\author[S. Grellier]{Sandrine Grellier}
\author[Z. He]{Zihui He} 
\address[P. Gérard]{Laboratoire de Math\'ematiques d'Orsay, Universit\'e Paris-Saclay,  91405 Orsay, France, CNRS}
\email{patrick.gerard@universite-paris-saclay.fr}
\address[S. Grellier]{Institut Denis Poisson, D\'epartement de Math\'ematiques, Universit\'e d'Orl\'eans, 45067 Orl\'eans Cedex 2, France}
\email{sandrine.grellier@univ-orleans.fr}
\address[Z. He]{Institute of Analysis, Karlsruhe Institute of Technology,
Englerstraße 2, 76131 Karlsruhe, Germany.}
\email{zihui.he@kit.edu}

\begin{document}
\date{\today}
\maketitle

\begin{abstract}
In this paper, we study the transfer of energy from low to high frequencies for a family of damped Szeg\H{o} equations. 
The cubic Szeg\H{o} equation has been introduced as a toy model for a totally non-dispersive degenerate Hamiltonian equation. It is a completely integrable system which develops growth of high Sobolev norms, detecting transfer of energy and hence cascades phenomena.

Here, we consider a two-parameter family of variants of the cubic Szeg\H{o} equation and prove that, adding a damping term unexpectedly promotes the existence of turbulent cascades. Furthermore, we give a panorama of the dynamics for such equations on a six-dimensional submanifold.  
\end{abstract}

\section{Introduction}
An interesting aspect of turbulence is the transfer of energy from long to short-wavelength modes, leading to concentration of energy on small spatial scales. It is usually quantified by growth of Sobolev norms. 
In this paper, we study turbulent cascades for the family of  damped Szeg\H{o} equations on the one-dimensional torus
\begin{equation}\label{DABS}
  i\d_t u+i\nu (u\vert \mathds{1}) = \Pi(|u|^2u)+\alpha (u\vert \mathds{1})-\beta S\Pi(|S^*u|^2S^*u),
\end{equation}
where $\nu>0$ and $\alpha,\beta\in \R $ are given parameters. The term $\nu$ is the damping term on the smallest Fourier mode
$ (u\vert \mathds{1})=\frac{1}{2\pi} \int_{-\pi}^{\pi} u(e^{ix})\, dx$.
The Szeg\H{o} projector $\Pi\colon L^2(\T)\to L^2_+(\T)$ is the Fourier multiplier defined by \begin{equation*}
    \Pi(u)=\sum_{k\ge 0} \hat u(k)e^{ikx}\in L^2_+(\T).
\end{equation*}
The shift operator $S:L^2_+(\T)\to L^2_+(\T)$  and its adjoint are defined by
\begin{equation*}
 Su=e^{ix}u\quad \text{and}\quad S^* u=e^{-ix}(u-(u\vert \mathds{1})).
\end{equation*} 
    When $\nu=\alpha=\beta=0$, we recover the usual cubic Szeg\H{o} equation 
\begin{equation}
\label{S}
  i\d_t u= \Pi(|u|^2u),
\end{equation}
which was introduced by the first two authors \cite{GG10} as a toy model of a non-dispersive Hamiltonian system.
The $\alpha$-deformation in \eqref{DABS} was first introduced by \textcite{Xu14} in the $\alpha$-Szeg\H{o} equation 
\begin{equation}\label{AS}
     i\d_t u= \Pi(|u|^2u)+\alpha (u\vert \mathds{1}),
\end{equation}
and the $\beta$-deformation in \eqref{DABS} was introduced by \textcite{BE20} in the $\beta$-Szeg\H{o} equation 
\begin{equation} \label{BS}
       i\d_t u = \Pi(|u|^2u)-\beta S\Pi(|S^*u|^2S^*u).
\end{equation}
Observe that, on the Fourier side, this equation corresponds to
$$i\frac{d \hat u(n)}{dt}=\sum_{\substack{m,k,l=0\\ n+m=l+k}} ^\infty C^{(\beta)}_{nmkl} \overline{\hat u(m)}\hat u(k)\hat u(l),\; n\in\N$$ where
\begin{eqnarray*}
C^{(\beta)}_{nmkl}=\left\{
\begin{array}{ll}
1&\text{ if }nmkl=0\cr
1-\beta &\text{ otherwise.}\cr
\end{array}
\right.
\end{eqnarray*}
When $\beta=1$, the equation \eqref{BS} is so called the truncated Szeg\H{o} equation and corresponds to the case where most of the interaction of the Fourier coefficients disappeared. Some other variants of the Szeg\H{o} equation have also been studied, see e.g. (\cite{Poc11,Thi19}).

\subsection{Promoted turbulence}
The Szeg\H{o} equation is a completely integrable Hamiltonian system with two Lax pair structures displaying some turbulence cascade phenomenon for a generic set of initial data, in spite of infinitely many conservation laws.   Namely, there exists a dense $G_\delta $ subset of initial data in $L^2_+\cap C^\infty $, such that the solutions of the cubic Szeg\H{o} equation satisfy, for every $s>\frac 12$,
$$\limsup_{t\to +\infty} \Vert u(t)\Vert _{H^s}=+\infty \ ,\ \liminf_{t\to +\infty} \Vert u(t)\Vert _{H^s}<+\infty\ .$$
Furthermore, this subset has an empty interior, since it does not contain any trigonometric polynomial \cite{GG17}. 

However, there is no explicit examples of such phenomenon, and even less is known about the existence of solutions with high-Sobolev norms tending to infinity.

Later the first two authors \cite{GG20} added a damping term to the cubic Szeg\H{o} equation \begin{equation}\label{DS}
  i\d_t u+i\nu (u\vert \mathds{1}) = \Pi(|u|^2u),\quad \nu>0.
\end{equation}
Paradoxically, the turbulence phenomenon is promoted by the damping term. Indeed, a nonempty open set of initial data generating trajectories with high-Sobolev norms tending to infinity was observed.
In comparison, this is not the case for the damped Benjamin--Ono equation, see \cite{Gas20}.

The goal of this paper is to generalise the study of turbulent cascades for the damped Szeg\H{o} equation to a family of damped Szeg\H{o} equations \eqref{DABS}.
Biasi and Evnin \cite{BE20} suggested the study of a two-parameter family of equations referred as the  $(\alpha,\beta)$-Szeg\H{o} equations given by
\begin{equation}\label{aabszego}
i\partial_t u=\Pi(|u|^2u)-\beta S\Pi(|S^*u|^2S^*u)+\alpha (u\vert \mathds{1}),\quad \alpha,\beta\in\R.
\end{equation}
Inspired by this, we study the damped $(\alpha,\beta)$-Szeg\H{o} equations \eqref{DABS}. 
The family is constructed in such a way that part of the Lax-pair structure inherited from the cubic Szeg\H{o} equation is preserved but the damping term breaks the Hamiltonian structure.
In our case, similar to the damped Szeg\H{o} equation, the damping term promotes the existence of unbounded trajectories. Our main result is the following.
\begin{theorem}\label{main}
 There exists an open subset $\Omega \subset H^{\frac 12}_+=H^{\frac 12}\cap L^2_+$ independent of $(\alpha ,\beta)$ such that, for every $s>\frac 12$, the set $\Omega \cap H^s_+$ is nonempty and, for every $\beta \ne 1$,  every solution $u$ of \eqref{DABS} with $u(0)\in \Omega \cap H^s_+$ satisfies
$$\Vert u(t)\Vert_{H^s}\underset{t\to \infty}{\longrightarrow}+\infty \ .$$
Furthermore, there exist rational initial data in $\Omega $ which generate  stationary solutions of \eqref{DABS}  for $\beta =1$.\end{theorem}
When $\beta$ is different from $1$, the damping term acts on the $(\alpha,\beta)$-Szeg\H{o} equations as on the cubic Szeg\H{o} equation\cite{GG20}. The  case $\beta=1$ of the damped truncated Szeg\H{o} equation appears to be more degenerate, and we do not know whether there exists a nonempty open subset of blowing up data. 

The wave turbulence phenomenon for Hamiltonian systems has been actively studied by mathematicians and physicists in the last decades.
Bourgain \cite{Bou00} asked whether there is a solution of the cubic defocusing nonlinear Schrödinger equation on the two-dimensional torus $\T^2$ with initial data $u_0\in H^s(\T^2)$, $s>1$, such that
$$\limsup_{t\to+\infty}\|u(t)\|_{H^s}=\infty.$$
There is still no complete answer to this question. However, the first mathematical evidence of such behaviour has been exhibited in the seminal work \cite{Col10} in which it is proven that, given any initial data with small Sobolev norm, it is possible to find a sufficiently large time for which the Sobolev norm of the solution is larger than any prescribed constant. 
This phenomenon also occurs for the half-wave equations on the real line or on the one-dimensional torus, see e.g. \cite{Poc13,GG121}.  Based on this, the first author,  Lenzmann, Pocovnicu, and Rapha\"el \cite{GLPR18} gave a complete picture for a class of solutions on the real line. Namely,  after the transient turbulence, the Sobolev norms of such solutions remains stationary large in time. 
The turbulence also occurs for two-dimensional incompressible Euler equations, the sharp double exponentially growing vorticity gradient on the disk was constructed by \textcite{KS14} and the existence of exponentially growing vorticity gradient solutions on the torus was shown by \textcite{Zla15}.

\subsection{Preliminary observations }
 For the damped $(\alpha,\beta)$-Szeg\H{o} equations \eqref{DABS}, the momentum 
$$\mathcal{M}(u)=(Du\vert u)=\sum_{k\ge 1}k|\hat u(k)|^2$$ is preserved by the flow. 
 An easy modification of the arguments in \cite{GG10} shows that \eqref{DABS} is globally wellposed 
on $H^s_+$ for every $s\ge \frac 12$. Our goal is to study the behaviour of solutions of \eqref{DABS} as $t\to +\infty $, in particular the growth of  $H^s$-Sobolev norms for $s>\frac 12$.

 The main property which allows us to do computations, is the existence of a Lax pair.  Namely, if $u$ satisfies \eqref{DABS} then
$$\frac{d}{dt}\tilde H_u=[C_u-\beta  B_{S^*u},\tilde H_u],$$
where $H_u$ is the Hankel operator and $\tilde H_u=S^*H_u$ is the shifted Hankel operator. The operators $B_u$ and $C_u$ are the anti-self-adjoint operators appearing in the Lax pairs of the cubic Szeg\H{o} equation (see Section 2.2 for the definitions).

Thanks to this Lax pair, there are invariant 
manifolds consisting of the functions $u$ such that $\text{rank}\,(\tilde H_u)=k$, $k\ge0$. From a well known result by Kronecker \cite{Kro}, these manifolds consist of the rational functions
$$u(x)=\frac{P_1(e^{ix})}{P_2(e^{ix})},$$
where $P_1$ and $P_2$ are polynomials of degrees at most $k$ with $\text{deg}(P_1)=k$ or $\text{deg}(P_2)=k$, no common roots and $P_2$ has no roots inside the disk $\{z\in\C\mid |z|<1\}$. 
\subsection{A special case}

In this section, we restrict ourselves to the lowest dimensional submanifold where $\tilde H_u$ has rank $1$
\begin{equation*}
\mathcal W:=\left \{u(x)=b+\frac {c\, {e}^{ix}}{1-p{e}^{ix}},\; b,c,p\in\C,  c\neq 0,\; |p|<1\right \}.
\end{equation*}
We will give a complete picture of \eqref{DABS} on $\cW$, which consists of  periodic, blow-up and scattering trajectories. 

We consider the trajectories with a fixed momentum $M>0$, and define 
\begin{equation*}
  \mathcal{E}_M=\{u\in\cW \mid  \mathcal{M}(u(t))=M\},\quad 
  \mathcal{C}_M=\{u(x)=c e^{ix} \mid |c|^2=M\}.
\end{equation*}
Notice that $\mathcal{C}_M\subset\mathcal{E}_M$ consists of the periodic trajectories.
 We write
\begin{equation*}
  \lS_{\nu,\alpha,\beta}(t)u_0=u(t)
\end{equation*}
for the solution of the equation \eqref{DABS}  with initial data $u_0\in H^{\frac12}_+(\T)$. Then we have the following theorem.
\begin{theorem} \label{thm:spe}
Let $\nu>0$ and $\alpha,\beta\in\R$. There exists a  three-dimensional submanifold $\Sigma_{M,\nu,\alpha,\beta}\subset \mathcal{E}_M$ disjoint from $\mathcal{C}_M$, invariant under the flow $\lS_{\nu,\alpha,\beta}(t)$ and such that $\Sigma_{M,\nu,\alpha,\beta}\cup \,\mathcal{C}_M$ is closed  and 
\begin{enumerate}
    \item If $u_0\in \mathcal{E}_M\setminus (\Sigma_{M,\nu,\alpha,\beta}\cup \,\mathcal{C}_M)$, then  
    \begin{equation} \label{blowup:0}
        \|\lS_{\nu,\alpha,\beta}(t)u_0 \|_{H^s}\sim_{s,\nu,\alpha,\beta,M}t^{s-\frac12}, \quad s>\frac12.
    \end{equation}
    \item If $u_0\in \Sigma_{M,\nu,\alpha,\beta}$, then 
    $\text{dist}\,(\lS_{\nu,\alpha,\beta}(t) u_0,\mathcal{C}_M)\sim e^{-ct},$ for some $c>0$.
    \end{enumerate}
\end{theorem}
Let us emphasize that on this submanifold $\cW$, unlike on higher dimensional submanifolds (see section \ref{StationaryOmega}), there is no difference between the case $\beta=1$ and the other cases. \\
Compared to Theorem \ref{main}, in Theorem \ref{thm:spe} the open set consisting of the initial data generating blow-up trajectories, is dense in $\cW$. Furthermore, the Sobolev norms of 
such generating trajectories grow at a uniform polynomial rate $\sim t^{s-\frac12}$, independent of $\nu,\alpha,\beta$.
This is consistent with the blow-up rate for the damped Szeg\H{o} equation \eqref{DS}, see \cite{GG20}.  
In contrast, the initial data in $\cW$ 
generate only bounded trajectories in the case of the Szeg\H{o} equation, the $\alpha$ and $\beta$-szeg\H{o} equations with negative $\alpha$ and $\beta$, see \cite{GG10,Xu14,BE20}.
However, if $\alpha>0$ and $\beta>0$, 
then even faster blow-up solutions occur for the $\alpha$ and $\beta$-Szeg\H{o} equations. In this case, there exist trajectories $u(t)$, whose Sobolev norms grow  exponentially in time with
\begin{equation}\label{exp:growth}
   \|u(t)\|_{H^s}\sim e^{ct(s-\frac12)}, \quad s> \frac12. 
\end{equation}
Moreover, if $\beta>9$, then there also exists a class of solutions for the $\beta$-Szeg\H{o} equation with the polynomially growing Sobolev norms at the rate \eqref{blowup:0}, see \cite{BE20}. In other words, the authors exhibit various strong turbulence phenomena for $\beta$-Szeg\H{o} equations when $\beta$ is large enough.

One important feature of the damped $(\alpha,\beta)$-Szeg\H{o} equations is the existence of the Lyapunov functional
\begin{equation*}
\frac{d}{dt}\|\lS_{\nu,\alpha,\beta}(t) u_0\|_{L^2}^2 +2\nu|(\lS_{\nu,\alpha,\beta}(t) u_0\vert \mathds{1})|^2=0.
\end{equation*}
Together with the conserved momentum, one infers the weak limit
points of $u(t)$ as $t\to+\infty$ in $H^{\frac12}$. In the second part of Theorem \ref{thm:spe}, when $\beta=1$, such weak limit points are also strong limit points. Namely, there exists $u_\infty\in \cC_M$ such that
\begin{equation*}
    \|\lS_{\nu,\alpha,1}(t)u_0-u_\infty\|_{H^s}\to0,\quad\forall s\ge\frac12.
\end{equation*}
 Let us complete this paragraph by a few more remarks about stationary solutions in the case $\beta =1$. On $\cW$, we already observed that these solutions are of the form $c e^{ix}$ with $c\ne 0$. An elementary  calculation shows that such initial data generate periodic solutions in the case $\beta \ne 1$ and arbitrary $(\alpha ,\nu )$. However, as stated in Theorem \ref{main}, one can construct rational stationary solutions in the case $\beta =1$ which generate blow up solutions for every $\beta \ne 1$, $\nu >0$ and arbitrary $\alpha \in \R$.



 \subsection{Organisation of this paper}
Section 2 is devoted to establishing some general properties of the undamped as well as the damped $(\alpha,\beta)$-Szeg\H{o} equations. Theorems \ref{main} and \ref{thm:spe} are proved in Section 3 and Section 4 respectively.

\section{Generalities on the damped and undamped \texorpdfstring{$(\alpha,\beta)$-Szeg\H{o}}{} equations} 
In the following, we denote by $t\mapsto \lS_{\alpha,\beta}(t)u_0$ (respectively  by  $t\mapsto \lS_{\nu,\alpha,\beta}(t)u_0$) the solution of the $(\alpha,\beta)$-Szeg\H{o} equation \eqref{aabszego} (respectively of the damped $(\alpha,\beta)$-Szeg\H{o} equation \eqref{DABS}) with initial datum $u_0$. 

\subsection{The Lyapunov functional}
As in the case of the damped Szeg\H{o} equation, an important tool in the study of  Equation (\ref{DABS}) is the existence of a Lyapunov functional. Precisely, the following lemma holds.
\begin{lemma}\label{limit}
Let $u_0\in H^{1/2}_+(\T)$. Then, for any $t\in \R$,
\begin{equation}\label{Lyapunov}
\frac d{dt}\Vert \lS_{\nu,\alpha,\beta}(t) u_0\Vert^2_{L^2}+2\nu |(\lS_{\nu,\alpha,\beta} u_0(t)\vert \mathds{1})|^2=0.
\end{equation} As a consequence, if $\nu>0$,
 $t\mapsto\Vert \lS_{\nu,\alpha,\beta}(t)u_0\Vert_{L^2}$ is decreasing, and $|(\lS_{\nu,\alpha,\beta}(t)u_0\vert \mathds{1})|$ is square integrable on $[0,+\infty)$, tending to zero as $t$ goes to $+\infty $.
\end{lemma}
\begin{proof}
Denote by $u(t):=\lS_{\nu,\alpha,\beta}(t)u_0$ the solution of \eqref{DABS} with $u(0)=u_0$. Observe first that $t\mapsto\Vert u(t)\Vert_{L^2}$ decreases:
\begin{eqnarray*}
\frac d{dt}\Vert u(t)\Vert^2_{L^2}&=& 2{\rm Re}(\partial_t u\vert u)=2{\rm Im} (i\partial_t u\vert u)\\
&=&2{\rm Im} (\Pi(|u|^2u)\vert u)-2\beta {\rm Im} (S\Pi(|S^*u|^2S^*u)\vert u) +2{\rm Im} ((\alpha-i\nu) ((u\vert \mathds{1})\vert u))\\
&=&-2\nu |(u(t)\vert \mathds{1})|^2
\end{eqnarray*}
Hence, $t\mapsto\Vert u(t)\Vert^2_{L^2}$ admits a limit at infinity and since
$$\Vert u(t)\Vert^2_{L^2}=\Vert u_0\Vert^2_{L^2}-2\nu\int_0^t|(u(s)\vert \mathds{1})|^2 ds$$  we deduce the finiteness of 
$$\int_0^\infty|(u(s)\vert \mathds{1})|^2 ds.$$
On the other hand, we claim that $$\frac d{dt}  |(u(t)\vert \mathds{1})|^2$$ is bounded. Indeed
\begin{eqnarray*}
\frac d{dt}  |(u(t)\vert \mathds{1})|^2&=& 2{\rm Re} ((\partial_t u\vert \mathds{1})(\mathds{1}\vert u))\\
&=& 2{\rm Im}((\alpha-i\nu) (u\vert \mathds{1})(\mathds{1}\vert u))+2{\rm Im}((\Pi(|u|^2u\vert \mathds{1})(\mathds{1}\vert u)))\\&-&2\beta {\rm Im}((S\Pi(|S^*u|^2S^*u\vert \mathds{1})(\mathds{1}\vert u)))\\
&=&-2\nu |(u\vert \mathds{1})|^2+2{\rm Im}((u^2\vert u)(\mathds{1}\vert u)).
\end{eqnarray*}
but $$|(u(t)\vert \mathds{1})|\le \Vert u\Vert_{L^2}$$
and 
$$|(u^2(t)\vert u(t))|\le \Vert u\Vert_{L^2}\times \Vert u\Vert_{L^4}^2\le \Vert u\Vert_{L^2}\times\Vert u\Vert_{H^{1/2}}^2\le  \Vert u_0\Vert_{L^2}(\mathcal M(u_0)+\Vert u_0\Vert_{L^2}^2).$$
From both observations, we conclude that $|(u(t)\vert \mathds{1})|$ tends to zero as $t$ goes to infinity.
\end{proof}

From Lemma \ref{limit} and the conservation of the momentum, the $H^{1/2}$ norm of $\lS_{\nu,\alpha,\beta}(t)u_0$ remains bounded as $t\to +\infty $, hence one can consider limit points $u_\infty$ of $\lS_{\nu,\alpha,\beta} (t)u_0$ for the weak topology of $H^{1/2}$ as $t\to +\infty $.
Another general lemma describes more precisely these limit points, according to  LaSalle's invariance principle.

\begin{proposition}\label{WeakLimit}
Let $u_0\in H^{1/2}(\T)$. \\Any $H^{1/2}$- weak limit point $u_\infty$ of $(\lS_{\nu,\alpha,\beta}(t) u_0)$ 
as $t\to +\infty$  satisfies $(\lS_{\nu,\alpha,\beta}(t)u_\infty\vert \mathds{1})=0$ for all $t$. In particular, $\lS_{\nu,\alpha,\beta}(t)u_\infty$ solves the $(\alpha,\beta)$- Szeg\H{o} equation -- in other words $\lS_{\nu,\alpha,\beta}(t)u_\infty=\lS_{\alpha,\beta}(t)u_\infty$.
\end{proposition}
\begin{proof}
Denote by $Q$ the limit of the decreasing non-negative function $$t\mapsto\Vert \lS_{\nu,\alpha,\beta}(t)u_0\Vert_{L^2}^2.$$ By the weak continuity of the flow in $H^{1/2}_+(\T)$,
$$u(t+t_n)=\lS_{\nu,\alpha,\beta}(t)u(t_n)\to \lS_{\nu,\alpha,\beta}(t)u_\infty$$ weakly in $H^{1/2}$ as $n\to \infty$. Hence, thanks to the Rellich theorem, $$\Vert u(t+t_n)\Vert_{L^2}^2\to \Vert \lS_{\nu,\alpha,\beta}(t)u_\infty\Vert ^2_{L^2}$$ as $n$ tends to infinity. On the other hand, by Lemma \ref{limit},  $$\Vert u(t+t_n)\Vert_{L^2}^2\to Q$$ so eventually, for every $t\in \R$,  $$\Vert \lS_{\nu,\alpha,\beta}(t)u_\infty\Vert ^2_{L^2}=\Vert u_\infty\Vert ^2_{L^2},$$ or $$\frac d{dt}\Vert \lS_{\nu,\alpha,\beta}(t)u_\infty\Vert^2_{L^2}=0.$$ 
Recall that, from \eqref{Lyapunov},
$$\frac{d}{dt}\Vert \lS_{\nu,\alpha,\beta}(t)u_\infty\Vert_{L^2}^2=-2\nu|(\lS_{\nu,\alpha,\beta}(t)u_\infty\vert \mathds{1})|^2.$$
It forces $(\lS_{\nu,\alpha,\beta}(t)u_\infty\vert \mathds{1})=0$ for all $t$. Hence $\lS_{\nu,\alpha,\beta}(t)u_\infty=\lS_{\alpha,\beta}(t)u_\infty$ is a solution to the $(\alpha,\beta)$-Szeg\H{o} equation without damping. 
\end{proof}
In order to characterise $u_\infty$, we need to recall some results about the cubic Szeg\H{o} equation corresponding to the case $\alpha=\beta=\nu=0$ in equation \eqref{DABS}.
\subsection{Hankel operators and the Lax pair structure}
\label{Hankel} 
In this paragraph, we recall some basic facts about Hankel operators and the special structure of the cubic Szeg\H{o} equation. We keep the notation of \cite{GG17} and we refer to it for details. For $u\in H^{\frac 12}_+$, we denote by $H_u$ the Hankel operator of symbol $u$ namely
$$H_u:\left\{\begin{array}{cll}
L^2_+(\T)&\to& L^2_+(\T)\\
f&\mapsto& \Pi(u\overline f)
\end{array}\right.$$
It is well known that, for $u$ in $H^{\frac 12}_+$, $H_u$ is Hilbert-Schmidt with $${\rm Tr}(H_u^2)=\sum_{k\ge 0} (k+1)|\hat u(k)|^2=\Vert u\Vert_{L^2}^2+\mathcal{M}(u).$$ One can also consider the shifted Hankel operator $\tilde H_u$ corresponding to $H_{S^*u}$ where $S^*$ denotes the adjoint of the shift operator $Sf(x):={e}^{ix}f(x)$.
This shifted Hankel operator is Hilbert-Schmidt as well, with  $${\rm Tr}(\tilde H_u^2)=\sum_{k\ge 0} k|\hat u(k)|^2=\mathcal{M}(u)$$ Observe in particular that 
\begin{equation}\label{L2norm}
\Vert u\Vert_{L^2}^2={\rm Tr}(H_u^2)-{\rm Tr}(\tilde H_u^2).
\end{equation}

A crucial property of the cubic Szeg\H{o} equation is its Lax pair structure. Namely, if $u$ is a smooth enough solution of the cubic Szeg\H{o} equation, then there exists two anti-selfadjoint operators $B_u, \, C_u$ such that 
\begin{equation}\label{LaxPair}\frac{d}{dt}H_u=H_{-i\Pi(|u|^2u)}=[B_u,H_u],\; \;\frac{d}{dt}\tilde H_u=\tilde H_{-i\Pi(|u|^2u)}=[C_u,\tilde H_u].\end{equation}
Classically, these equalities imply that $H_{u(t)}$ and $\tilde H_{u(t)}$ are isometrically equivalent to $H_{u(0)}$ and $\tilde H_{u(0)}$ (see \cite{GG17} for instance). In particular, both spectra of $H_u$ and $\tilde H_u$ are preserved by the cubic Szeg\H{o} flow. It motivated the study of the spectral properties of both Hankel operators that we recall here.\\
For $u\in H^{\frac 12}_+$, let $(s_j^2)_{j\ge 1}$ be the strictly decreasing sequence of positive eigenvalues of $H_u^2$ and $\tilde H_u^2$. Following the terminology of \cite{GG17}, we say that 
\begin{itemize}
\item $\sigma^2$ is an $H$-dominant eigenvalue if $$\ker( H_u^2-\sigma^2I)\neq\emptyset\text{ and  }
u\not\perp \ker( H_u^2-\sigma^2I).$$
Respectively, 
\item $\sigma^2$ is an $\tilde H$-dominant eigenvalue if $$\ker(\tilde H_u^2-\sigma^2I)\neq\emptyset\text{ and }u\not\perp \ker(\tilde H_u^2-\sigma^2I).$$
\end{itemize}

 From the fundamental property $\tilde H_u^2=H_u^2-(\cdot\vert u)u$ and the min-max formula, the first two authors proved that the $s_{2j-1}^2$ correspond to $H$-dominant eigenvalues of $H_u^2$ while the $s_{2j}^2$ correspond to $\tilde H$-dominant eigenvalues of $\tilde H_u^2$.  Furthermore,  the eigenvalues of $H_u^2$ and of $\tilde H_u^2$ interlace and, as a consequence, if 
$m_j:=\dim \ker (H_u^2-s_j^2I)$ and  $\tilde m_j:=\dim \ker (\tilde H_u^2-s_j^2I)$,  then $$m_{2j-1}=\tilde m_{2j-1}+1\; \text{ as }\tilde m_{2j}=m_{2j}+1.$$ \\
To complete the spectral analysis of these Hankel operators, we need to recall the notion of Blaschke product. A function $b$ is a Blaschke product of degree $m$ if $$b(x)={e^{i\varphi}}\prod_{j=1}^m \frac{{e^{ix}}-p_j}{1-\overline {p_j}{e^{ix}}}$$ for some $p_j\in \C$ with $|p_j|<1$, $j=1$ to $m$.
As proved  in \cite{GG17} --- see also \cite{GP} for a generalisation to non compact Hankel operators---, for any $H$-dominant eigenvalue $s_{2j-1}^2$, there exists a Blaschke product $\Psi_{2j-1}$ of degree $m_{2j-1}-1$ such that, if $u_j$ denotes the orthogonal projection of $u$ on the eigenspace $\ker(H_u^2-s_{2j-1}^2I)$, then $$\Psi_{2j-1}H_u(u_j)=s_{2j-1}u_j.$$
Analogously, for any $\tilde H$-dominant eigenvalue $s_{2k}^2$, there exists a Blaschke product $\Psi_{2k}$ of degree $\tilde m_{2k}$ such that, if $\tilde u_k$ denotes the orthogonal projection of $u$ on $\ker(\tilde H_u^2-s_{2k}^2I)$ then $$\tilde H_u(\tilde u_k)=\Psi_{2k}s_{2k}\tilde u_k.$$

The first two authors proved in \cite{GG17} that the sequence $((s_j^2),(\Psi_j))$ characterises $u$, and that it provides a system of action-angle variables for the Hamiltonian evolution.  Namely, 
if $u(0)$ has spectral coordinates $((s_j^2),(\Psi_j))$ then $u(t)$ has spectral coordinates $$((s_j^2),({e}^{i(-1)^js_j^2t}\Psi_j)).$$ 
We are now in position to characterise the asymptotics of the damped $(\alpha,\beta)$-Szeg\H{o} equation. As claimed in the introduction, the equation inherits one Lax pair from the Szeg\H{o} equation related to the shifted Hankel operator $\tilde H_u$. Recall that, from the definition of $\tilde H_{v}=H_{S^*v}$, on one hand, the shifted Hankel operator associated to a constant symbol is identically $0$ and, on the other hand, $\tilde H_{Sv}=H_{SS^*v}=H_v$. Hence, using \eqref{LaxPair}, one obtains for $u(t):=\lS_{\nu,\alpha,\beta}(t) u_0$, 
\begin{eqnarray*}
\frac d{dt}\tilde H_u&=&\tilde H_{-i\Pi(|u|^2u)+i\beta S\Pi(|S^*u|^2S^*u)}+(\nu-i\alpha) \tilde H_{(u\vert  \mathds{1})}\\
&=&[C_u,\tilde H_u]+\beta H_{i\Pi(|S^*u|^2S^*u)}\\
&=&[C_u,\tilde H_u]-\beta[B_{S^*u}, H_{S^*u}]
\end{eqnarray*}
Hence, 
\begin{equation}\label{LaxK}
\frac d{dt}\tilde H_u=[C_u-\beta B_{S^*u},\tilde H_u]
\end{equation}
where $B_u $ and $C_u$ are the anti-selfadjoint operator given by
\begin{equation}
B_u=-iT_{|u|^2}+\frac i2H_u^2\text{ and }C_u=-iT_{|u|^2}+\frac i2\tilde H_u^2.
\end{equation}
Here $T_b$ denotes the Toeplitz operator of symbol $b$ given on $L^2_+$ by $T_b(f)=\Pi(bf)$.
As a usual consequence, $\tilde H_{\lS_{\nu,\alpha,\beta}(t) u_0}$ is unitarily equivalent to $\tilde H_{u_0}$ and for instance, the class of symbol $u$ with $\tilde H_u$ of fixed finite rank is preserved by the damped $(\alpha,\beta)$-Szeg\H{o} flow. In particular \begin{equation}\label{SetW}
\mathcal W:=\left \{u(x)=b+\frac {c\, {e}^{ix}}{1-p{e}^{ix}},\; b,c,p\in\C,  c\neq 0,\; |p|<1\right \}
\end{equation}
 is invariant by the flow since it corresponds to the set of symbol whose shifted Hankel operators are of rank $1$.\\
 Another consequence is the following result, where the difference between the cases $\beta \neq 1$ and $\beta =1$ clearly appears.
\begin{theorem}\label{u(0)=0}
(i) Assume $\beta\neq 1$.\\
The solutions $\lS_{\alpha,\beta}(t)u_0$ of the $(\alpha,\beta)$-Szeg\H{o} equation satisfying $(\lS_{\alpha,\beta}(t)u_0\vert \mathds{1})=0$ for all $t$ are characterised by the fact that all the $H$-dominant eigenvalues of $u_0$ are at least of multiplicity $2$ and hence, are eigenvalues of $\tilde H_{u_0}^2$. Furthermore, if $\{\sigma_k^2\}_k$ denotes the strictly decreasing sequence of the eigenvalues of $\tilde H_{u_0}^2$, one has
$$\Vert \lS_{\alpha,\beta}(t)u_0\Vert_{L^2}^2=\sum_k(-1)^{k-1}\sigma_k^2.$$
(ii) In the case $\beta =1$, the solutions $\lS_{\alpha, 1}(t)u_0$ of the $(\alpha, 1)$-Szeg\H{o} equation satisfying $(\lS_{\alpha, 1}(t)u_0\vert \mathds{1})=0$ for all $t$ are characterised by 
$$(u_0\vert 1)=(H_{u_0}^2(u_0)\vert 1)=0\ ,$$
and are stationary solutions.\end{theorem}
\begin{proof}
(i) Write $u:=\lS_{\alpha,\beta}(\cdot)u_0$. The scheme of the proof is to observe that the closure of the vector space $<u>_{H^2_u}$ spanned by the $H_u^{2m}(u)$, $m\in\N$ is orthogonal to $\mathds{1}$. First, as $$H^2_u=\tilde H^2_u+(\cdot\vert u)u,$$ 
this space equals the closure of the vector space $<u>_{\tilde H^2_u}$ spanned by the $\tilde H_u^{2m}(u)$, $m\in\N$. 
We prove by induction on $m$ that 
\begin{equation}\label{induction} \forall t\in \R \ ,\ 
(\tilde H^{2m}_u(u)\vert \mathds{1})(t)=0.
\end{equation}
For $m=0$, it follows from the assumption. For larger $m$, we observe that $\tilde H_u^2$ is self-adjoint, that $|S^*u|^2=|u|^2$ and that $\tilde H^2_u=H^2_u$ as $(u\vert \mathds{1})=0$ by assumption.\\
For $m=1$, we write
\begin{eqnarray*}
0=i(\partial_t u\vert \mathds{1})&=&(\Pi(|u|^2u)\vert \mathds{1})-\beta(S\Pi(|S^*u|^2S^*u)\vert \mathds{1})\\
&=&(\Pi(|u|^2u)\vert \mathds{1})=(u\vert H^2_u(\mathds{1}))=(u\vert \tilde H^2_u(\mathds{1}))=(\tilde H^2_u(u)\vert \mathds{1}).
\end{eqnarray*}

Assume now that 
$(\tilde H^{2j}_u(u)\vert \mathds{1})\equiv 0$ for any $j\le m$ and let us prove $$(\tilde H^{2(m+1)}_u(u)\vert \mathds{1})\equiv 0.$$

We write
\begin{eqnarray*}
0&=&\frac{d}{dt}(\tilde H^{2m}_u(u)\vert \mathds{1})=(\tilde H_u^{2m}(-iT_{|u|^2}u+i\beta ST_{|S^*u|^2} S^*u)\vert \mathds{1})\\
&+&([C_u-\beta B_{S^*u},\tilde H_u^{2m}]u\vert \mathds{1})\\
&=&(\tilde H_u^{2m}(-iT_{|u|^2}u+i\beta ST_{|u|^2} S^*u)\vert \mathds{1})+([-iT_{|u|^2}+i\beta T_{|u|^2},\tilde H_u^{2m}](u)\vert \mathds{1})\\
&=&i(\beta-1)(T_{|u|^2}(\tilde H_u^{2m})(u))\vert \mathds{1})=i(\beta-1)(\tilde H_u^{2m}(u)\vert |u|^2)\\
&=&i(\beta-1)(\tilde H_u^{2m}(u)\vert H_u^2(\mathds{1}))=i(\beta-1)(\tilde H_u^{2m}(u)\vert \tilde H_u^2(\mathds{1}))\\
&=&i(\beta-1)(\tilde H_u^{2(m+1)}(u)\vert \mathds{1})
\end{eqnarray*}
Here we used the property that $S\Pi S^*(f)=\Pi(f)-(f\vert \mathds{1})$ so that $ST_{|u|^2}S^*u=T_{|u|^2}u$. 
Eventually we get \eqref{induction}. 

Observe that, for any $H$-dominant eigenvalue $s_{2j-1}^2$, the orthogonal projection $u_j$ of $u$ onto the eigenspace $\ker(H_{u}^2-s_{2j-1}^2I)$ belongs to the vector space $<u>_{H^2_u}$. From the preceding result, this space is orthogonal to $\mathds{1}$ hence $u_j$ is orthogonal to $\mathds{1}$. By the spectral analysis of the Hankel operator recalled above, there exists a Blaschke product $\Psi_{2j-1}$ of degree $m_{2j-1}-1$ with  $s_{2j-1} u_j=\Psi_{2j-1} H_{u}(u_j)$. Taking the scalar product with $\mathds{1}$ gives $$0=\Psi_{2j-1}(0)(H_{u}(u_j)\vert \mathds{1})=\Psi_{2j-1}(0)\Vert u_{j}\Vert_{L^2},$$ hence $\Psi_{2j-1}(0)=0$ and the degree of $\Psi_{2j-1}$ is at least $1$. It follows that the multiplicity of $s_{2j-1}^2$ is at least $2$ (recall that from the definition of $u_j$, $\Vert u_j\Vert_{L^2}\neq 0$).
From the interlacement property, this eigenvalue is also an eigenvalue for $\tilde H_{u}^2$. Let $\{\sigma_k^2\}_k$ denote the strictly decreasing sequence of the eigenvalues of $\tilde H_{u}^2$. We denote by $m_k$ the multiplicity of  $\sigma_k^2$ as an eigenvalue of $H_u^2$ and by $\tilde m_k$ its multiplicity as an eigenvalue of $\tilde H_{u}^2$. From the interlacement property, if $k$ is odd, $m_k=\tilde m_k+1$ and if $k$ is even, $m_k=\tilde m_k-1$. We now compute the $L^2$ norm of $\lS_{\alpha,\beta}(t)u_0$ using \eqref{L2norm}:
$$\Vert \lS_{\alpha,\beta}(t)u_0\Vert_{L^2}^2={\rm Tr }H_{u}^2-{\rm Tr}\tilde H_{u}^2=\sum m_k\sigma_k^2-\sum \tilde m_k \sigma_k^2=\sum_k(-1)^{k-1}\sigma_k^2.$$
(ii) Next we assume $\beta =1$. Assume $u_0$ is  such that $(\lS_{\alpha,1}(t)u_0\vert 1)=0$ for every $t$. Then obviously $(u_0\vert 1)=0$, and the above calculation for $m=1$ implies that 
$(H_{u_0}^2(u_0)\vert 1)=0$. Conversely, if $(u_0\vert 1)=0$ and $(H_{u_0}^2(u_0)\vert 1)=0$, then, as already observed, $|u_0|^2=|S^*u_0|^2$, and $ST_{|u_0|^2}S^*u_0=T_{|u_0|^2}u_0$ ,
hence $\lS_{\alpha ,1}(.)u_0=u_0$.

\end{proof}

Once these basic properties are established, we are in position to prove Theorem \ref{main}.

\section{Exploding trajectories in the case \texorpdfstring{$\beta\neq1$}{}}

In this section, we consider trajectories of \eqref{DABS} in $H^s$, $s>\frac 12$, along which the $H^s$ norm of $u(t)$ tends to infinity as $t\to +\infty$. 
Let us define the functional $$F(u)=\sum_k(-1)^{k-1}\sigma_k^2$$ where $(\sigma_k^2)_k$ is the strictly decreasing sequence of positive eigenvalues of $\tilde H_{u}^2$. We prove the following result, which is very similar to the case of the damped Szeg\H{o} equation \cite{GG20}.
\begin{theorem}\label{Exploding}
Assume $\beta\neq 1$. Let $s>\frac 12$. If $u_0\in H^s_+$ satisfies
\begin{itemize}
\item either $ \Vert u_0\Vert_{L^2}^2 <F(u_0) $,
\item or $ \Vert u_0\Vert_{L^2}^2 =F(u_0)$ and $(u_0\vert  \mathds{1})\ne 0$,
\end{itemize}
then the  $H^s$--norm of the solution of the damped $(\alpha,\beta)$-Szeg\H{o} equation $$\Vert u(t)\Vert_{H^s}=\Vert \lS_{\nu,\alpha,\beta}(t) u_0\Vert_{H^s}$$ tends to $+\infty$ as $t$ tends to  $+\infty$.
\end{theorem}
\begin{remark}
Let us emphasize that the conditions driving to unbounded Sobolev trajectories are independent on $\alpha,\beta$.
\end{remark}
\begin{proof}
Let us proceed by contradiction and assume that there exists a sequence $t_n\to +\infty $ such that $u(t_n):=\lS_{\nu,\alpha,\beta}(t_n)u_0$ is bounded in $H^s$. 
We may assume that $u(t_n)$ is weakly convergent to some $u_\infty $ in $H^s_+$. By the Rellich theorem, the convergence is strong in $H^{\frac 12}_+$,
and 
\begin{equation}\label{MM}
\mathcal{M}(u_\infty )=\mathcal{M}(u_0)=\sum_{\sigma^2 \in \Sigma (u_0)}m(\sigma )\sigma ^2\ ,
\end{equation}
where $\Sigma(u_0)$ denotes the set of eigenvalues of $\tilde H_{u_0}^2$ and $m(\sigma )$ the multiplicity of $\sigma ^2\in \Sigma(u_0)$. By the Lax pair structure, the eigenvalues of $\tilde H_{u(t_n)}^2$ are the same as the eigenvalues of $\tilde H_{u_0}^2$, with the same multiplicities, hence every eigenvalue $\sigma^2 $ of $\tilde H_{u_\infty}^2$ must belong to $\Sigma (u_0)$, with a multiplicity not bigger than $m(\sigma )$. In view of  identity  \eqref{MM}, we infer that
$$\Sigma (u_\infty )=\Sigma (u_0)\ ,$$
with the same multiplicities. On the other hand, from Proposition \ref{WeakLimit}, we know that $u_\infty$ generates a solution of the cubic $(\alpha,\beta)$-Szeg\H{o} equation which is orthogonal to $1$ at every time. Consequently, Theorem \ref{u(0)=0} gives
$$\| u_\infty \|_{L^2}^2=F(u_0)\ .$$
Since the $L^2$-norm of the solution is decreasing by Lemma \ref{WeakLimit}, $ \| u_0\|_{L^2}^2\ge \| u_\infty \|_{L^2}^2\ .$  Hence, $ \| u_0\|_{L^2}^2\ge F(u_0)$.
 If $ \| u_0\|_{L^2}^2=F(u_0)$ then $\Vert \lS_{\nu,\alpha,\beta}(t)u_0\Vert _{L^2}$ remains constant and necessarily, by the Lyapunov functional identity \eqref{Lyapunov}, $(\lS_{\nu,\alpha,\beta} (t)u_0\vert  \mathds{1})=0$ so that in particular, $(u_0\vert  \mathds{1})=0$. Hence, the case $ \| u_0\|_{L^2}^2=F(u_0)$ and $(u_0\vert  \mathds{1})\neq 0$ drives to an exploding orbit in $H^s$ as well as the case $ \| u_0\|_{L^2}< F(u_0)$. It ends the proof of Theorem \ref{Exploding}. 
\end{proof}

\subsection{An open condition}
As a  corollary of Theorem  \ref{Exploding}, we get the following result, which implies the first part of Theorem \ref{main}.
\begin{corollary} Assume $\beta \neq 1$ Denote by $\Omega $ the interior in $H^{\frac 12}_+$ of the set of $u_0\in H^{\frac 12}_+$ such that $ \Vert u_0\Vert_{L^2}^2 <F(u_0) $. 
For every $s>\frac 12$,  $\Omega \cap H^s_+$ of $H^{s}_+(\T)$ is not empty, and 
every solution $u$ of \ref{DABS} with $u(0)\in\Omega \cap H^s_+$ satisfies
$$\Vert u(t)\Vert_{H^s}\to \infty$$
as $t$ tends to $+\infty $. 
\end{corollary}
\begin{proof}
By elementary perturbation theory, it is easy to prove that 
function $F$ is continuous at those $u$ of $H^{1/2}_+(\T)$ such that $\tilde H_u^2$ has simple non zero spectrum. Furthermore, in the
particular case
$$u(x)=\frac{{e}^{ix}}{1-p\, {e}^{ix}}, \; |p|<1,\; p\neq 0,$$ it is easy to check that $\tilde H_u^2$ has rank one with $\frac 1{(1-|p|^2)^2}$ as simple eigenvalue.
As $\Vert u\Vert_{L^2}^2=\frac 1{1-|p|^2}$, this function belongs to $\Omega $, and moreover it belongs to every $H^s$. 
In view of Theorem \ref{Exploding}, this completes the proof.
\end{proof}

We give a simple class of functions in $\Omega$.

\begin{example}
The set of functions $u_0$ whose nonzero eigenvalues of $H_{u_0}^2$ and $\tilde H_{u_0}^2$ are all simple, and form the decreasing square summable list
$$\rho_1>\sigma_1>\rho_2>\sigma_2>\dots $$
with
$$\sum_j \rho_j^2<2\sum_{k\, {\rm odd}}\sigma_k^2$$
is a subset of $\Omega$.
\end{example}
\subsection{Elements of \texorpdfstring{$\Omega$}{} generating stationary solutions for \texorpdfstring{$\beta=1$}{}}\label{StationaryOmega}

Finally, we complete the proof of Theorem \ref{main} by constructing rational elements in $\Omega$ which generate stationary solutions of the truncated Szeg\H{o} equation. This example illustrates the contrast of the case $\beta=1$ with the other cases. More specifically, we give an example of a rational datum $u$ as in the above example and  satisfying
\begin{equation}\label{condition}
(u\vert 1)=(H_u^2(u)\vert 1)=0\ .
\end{equation} From \cite{GG17}, there exists a system of coordinates in which any rational function is characterized by a finite sequence $(s_{j}^2,\Psi_j)_j$ (see paragraph \ref{Hankel}) and conversely, to any such finite sequence corresponds a unique rational function.
We seek $u$ such that $H_u$ has three simple singular values $\rho_1,\rho_2,\rho_2$, and $\tilde H_u$ has three simple singular values $\sigma_1,\sigma_2,\sigma_3$ with 
$$\rho_1>\sigma_1>\rho_2>\sigma_2>\rho_3>\sigma_3=0\ .$$
Denote by $u_j, j=1,2,3$ the projection of $u$ onto $\ker (H_u^2-\rho_j^2)$. Then we know that there exist $\varphi_j\in \T $ such that
$$e^{i\varphi_j}H_u(u_j)=\rho_ju_j \ .$$
Consequently,
$$(u\vert 1)=\sum_{j=1}^3 \frac{1}{\rho_j}e^{i\varphi_j}\Vert u_j\Vert_{L^2}^2\ ,\ (H_u^2(u)\vert 1)=\sum_{j=1}^3 \rho_je^{i\varphi_j}\Vert u_j\Vert_{L^2}^2\ .$$
Choosing $\varphi_1=\varphi_3=0\ ,\ \varphi_2=\pi \ ,$
the three conditions \eqref{condition} read
\begin{eqnarray*}
\frac{1}{\rho_2}\Vert u_2\Vert_{L^2}^2&=&\frac{1}{\rho_1}\Vert u_1\Vert_{L^2}^2+\frac{1}{\rho_3}\Vert u_3\Vert_{L^2}^2\ ,\\
\rho_2\Vert u_2\Vert_{L^2}^2&=&\rho_1\Vert u_1\Vert_{L^2}^2+\rho_3\Vert u_3\Vert_{L^2}^2\ ,\\
\rho_1^2+\rho_2^2+\rho_3^2&<&2\sigma_1^2\ .
\end{eqnarray*}
Recall that
$$\Vert u_j\Vert_{L^2}^2=\frac{(\rho_j^2-\sigma_1^2)(\rho_j^2-\sigma_2^2)\rho_j^2}{\prod_{1\leq k\ne j\leq 3}(\rho_j^2-\rho_k^2)}\ .$$
Therefore we can reformulate the problem as finding $$\rho_1>\sigma_1>\rho_2>\sigma_2>\rho_3>0$$ such that
\begin{eqnarray*}
\rho_1(\rho_1^2-\sigma_1^2)(\rho_1^2-\sigma_2^2)&=&-\rho_2(\rho_2^2-\sigma_1^2)(\rho_2^2-\sigma_2^2)=\rho_3(\rho_3^2-\sigma_1^2)(\rho_3^2-\sigma_2^2)\ ,\\
\rho_1^2+\rho_2^2+\rho_3^2&<&2\sigma_1^2\ .
\end{eqnarray*}
Let us fix $\sigma_1, \sigma_2$ such that $\sigma_1>\sigma_2>0$, and set
$$P(x):=x(x^2-\sigma_1^2)(x^2-\sigma_2^2)\ .$$
Notice that 
$$P'(0)=\sigma_1^2\sigma_2^2>0\ ,\ P'(\sigma_2)=2\sigma_2^2(\sigma_2^2-\sigma_1^2)<0\ ,\ P'(\sigma_1)=2\sigma_1^2(\sigma_1^2-\sigma_2^2)> 0\ .$$
By the inverse function theorem, for every $\eps >0$ small enough, there exist  $\rho_1(\eps )$ in a $\eps $--neighborhood of $\sigma_1$,
$\rho_2(\eps )$ in a $\eps $--neighborhood of $\sigma_2$, $\rho_3(\eps )$
in a $\eps $--neighborhood of $0$, such that
$$P(\rho_1(\eps ))=-P(\rho_2(\eps ))= P(\rho_3(\eps ))=\eps \ .$$
Furthermore, 
$$\rho_1(\eps )>\sigma_1>\rho_2(\eps )>\sigma_2>\rho_3(\eps )>0\ .$$
Consequently,
$$\rho_1(\eps )^2+\rho_2(\eps )^2+\rho_3(\eps )^2=\sigma_1^2+\sigma_2^2+O(\eps )<2\sigma_1^2$$

if $\eps $ is small enough. This completes the proof.

\section{A special case of \texorpdfstring{$\cW$}{}}
In this section, we provide a panorama of the dynamics of the damped $(\alpha,\beta)$-Szeg\H{o} equations for any fixed $(\alpha,\beta)\in \R^2$ on the  six-dimensional submanifold 
\begin{equation*}
\mathcal W:=\left \{u(x)=b+\frac {c\, {e}^{ix}}{1-p{e}^{ix}},\; b,c,p\in\C,  c\neq 0,\; |p|<1\right \}.
\end{equation*}
Recall that $\cW$ is preserved by the damped $(\alpha,\beta)$-Szeg\H{o} flow since it corresponds to the symbol of the shifted Hankel operator of rank 1.

For $u\in\cW$, we calculate the mass and the conserved momentum as follows
\begin{align*}
    \|u\|_{L^2}^2&=|b|^2+\frac{|c|^2}{1-|p|^2}, \quad  \mathcal{M}(u)=\frac{|c|^2}{(1-|p|^2)^2}.
\end{align*}
We will repeatedly use the relation between the mass and the momentum 
\begin{equation}\label{fact-1}
    \|u\|_{L^2}^2=|b|^2+\mathcal{M}(u)(1-|p|^ 2).
\end{equation}
We will consider the solutions of \eqref{DABS} with a  fixed momentum $\mathcal{M}(u(t))=M>0$. We define two subsets of $\cW$
\begin{equation*}
  \mathcal{E}_M=\{u\in\cW \mid  \mathcal{M}(u)=M\},\quad 
  \mathcal{C}_M=\{u(x)=ce^{ix} \mid |c|^2=M\}.
\end{equation*}
We observe that $\mathcal C_M\subset\mathcal{E}_M$ is invariant under the damped $(\alpha,\beta)$-flow \eqref{DABS} which consists of the periodic trajectories.

We write the damped $(\alpha,\beta)$-Szeg\H{o} equations on $\mathcal{E}_M$  in the $(b,c,p)$-coordinate as 
\begin{equation}
\label{ODE:bcp}
\left\{
\begin{aligned}
ib'+(i\nu -\alpha)b&=(|b|^2+2M(1-|p|^2))b+Mc\overline{p},\\
ic'&=2|b|^2c+2M(1-|p|^2)bp+(1-\beta)Mc,\\
ip'&=c\overline{b}+(1-\beta)Mp(1-|p|^2),
\end{aligned}
\right.
\end{equation}
where $\nu>0$ is the coefficient of the damping term in \eqref{DABS}.

We will determine all types of trajectories of the damped $(\alpha,\beta)$-Szeg\H{o} equations on $\mathcal{E}_M\setminus \mathcal C_M\subset\cW$, which consists of blow-up and scattering trajectories. By Lemma \ref{limit}, the $L^2$ norm of $u(t)$ converges 
\begin{align*}
    \|u(t)\|_{L^2}^2=|b(t)|^2+M(1-|p(t)|^2)\to Q\quad\text{and}\quad  |b(t)|=|(u(t)\vert \mathds 1)|\to 0,
\end{align*}
as $t\to+\infty$,
which implies $M(1-|p(t)|^2)\to Q$. As a consequence, $|p(t)|$ admits a limit in $[0,1]$. We claim that this limit can only be $0$ or $1$, which corresponds to the scattering trajectories or the blow-up trajectories respectively.

We prove that the limit of $|p(t)|$ can only be $0$ or $1$ by contradiction. Indeed, if  $0<\lim_{t\to+\infty}|p(t)|^2<1$, then the trajectory $\{u(t)\}$ is compact in $H^s(\T)$. As a consequence, $u(t)$ has a weak limit $u_\infty\in\cW$. By Proposition \ref{WeakLimit}, $$\lS_{\nu,\alpha,\beta}(t) u_\infty=b_\infty(t)+\frac{c_\infty(t)e^{ix}}{1-p_\infty(t)e^{ix}}$$
is a solution of the $(\alpha,\beta)$-Szeg\H{o} equation \eqref{aabszego} with $(\lS_{\nu,\alpha,\beta}(t) u_\infty\vert\mathds{1})=b_\infty(t)=0 $. Moreover, the triplet $(0,c_\infty,p_\infty)$  satisfies the ODE system $\eqref{ODE:bcp}_1$, which implies
$$Mc_\infty(t)\overline{p_\infty}(t)=0.$$
On the other hand, the momentum conservation law $$\mathcal{M}(\lS_{\nu,\alpha,\beta}(t)u_\infty)=\mathcal{M}(u_\infty)=\mathcal{M}(u_0)>0$$ 
ensures that $c_\infty$ cannot be $0$. Therefore, $p_\infty=0$,
which contradicts our assumption.

We will show that all the initial values $u_0$ with corresponding $|p(t)|\to1$ form a dense open set of $\cW$, on which the growth of the $H^s$ norm of $ \lS_{\nu,\alpha,\beta}(t)u_0$ is of order $t^{s-\frac12}$ as $t\to+\infty$. We remark that, as a consequence of Theorem \ref{Exploding}, those initial values $u_0$ satisfy 
\begin{equation}
\label{CondExplode}
\| \lS_{\nu,\alpha,\beta}(t)u_0\|_{L^2}^2<F(u_0)
\end{equation}  
for some $t$ in the case $\beta\neq1$. And $F(u_0)=\mathcal{M}(u_0)=M$ when $u_0\in\mathcal W$, so that condition \eqref{CondExplode} reads 
\begin{equation*}
\Vert \lS_{\nu,\alpha,\beta}(t)u_0\Vert_{L^2}^2<M    
\end{equation*}
for some $t$. 

We will show that the case $|p(t)|\to0$ corresponds to trajectories which exponentially converge to $\mathcal{C}_M$. The conserved momentum $\mathcal{M}(u)=\frac{|c|^2}{(1-|p|^2)^2}=M$ implies that
$$|c(t)|^2\to M.$$ 
On the other hand, the decay of $|b(t)|$ (showed in Lemma \ref{limit}) implies that
\begin{equation*}
    \|u(t)\|_{L^2}^2=|b(t)|^2+M(1-|p(t)|^2)\to M.
\end{equation*}
Since $\|u(t)\|_{L^2}^2$ decays monotonically, to study the non-periodic trajectories corresponding to $|p(t)|\to0$, one only needs to study the trajectories $\{u(t)\}$ disjoint from $\cC_M$ and 
$$\|u(t)\|_{L^2}^2\ge M,\quad \forall t\ge0.$$ 

The following theorem of the alternative holds:
\begin{theorem} \label{twocases}
Let $\nu>0$ and $\alpha,\beta\in\R$. Then there exists a three dimensional submanifold $\Sigma_{M,\nu,\alpha,\beta}\subset \mathcal{E}_M$, disjoint from $\mathcal{C}_M$ and invariant under the flow $\lS_{\nu,\alpha,\beta}(t)$, such that $\Sigma_{M,\nu,\alpha,\beta}\cup \,\mathcal{C}_M$ is closed  and the following holds:
\begin{enumerate}
    \item If $u_0\in \mathcal{E}_M\setminus (\Sigma_{M,\nu,\alpha,\beta}\cup \,\mathcal{C}_M)$, then $\|\lS_{\nu,\alpha,\beta}(t)u_0 \|_{H^s}^2$ blows up with the rate  \begin{equation}\label{blowup}
        \|\lS_{\nu,\alpha,\beta}(t)u_0 \|_{H^s}^2\sim a^2t^{2s-1},\quad s>\frac12,
    \end{equation}
    where 
    \begin{equation*}
    a^2(s,\nu,\alpha,\beta,M)
    =\Gamma(2s+1)M^{4s-1}\left(\frac{\nu^2+((1-\beta)M-\alpha)^2}{2\nu}\right)^{1-2s} .  
    \end{equation*}
    \item If $u_0\in \Sigma_{M,\nu,\alpha,\beta}$, then $\lS_{\nu,\alpha,\beta}(t) u_0$ tends to $\mathcal{C}_M$ as $t\to+\infty$, and 
    \begin{equation*}
        \text{dist}\ (\lS_{\nu,\alpha,\beta}(t) u_0,\mathcal{C}_M)\sim e^{-\frac{\nu+\sigma}{2}t},
    \end{equation*}
    where $\sigma=\left(\frac{(\nu^2-\alpha^2-4M\alpha)+\sqrt{(\nu^2-\alpha^2-4M\alpha)^2+4\nu^2( 2M+\alpha)^2}}{2}\right)^{\frac12}\ge \nu$.
    \end{enumerate}
\end{theorem}

{ This alternative behavior holds for all $(\nu,\alpha,\beta)$, which is consistent with the dynamics of the damped Szeg\H{o} equation $(\alpha=\beta=0)$. Indeed, we will follow a similar argument as for the damped Szeg\H{o} equation \cite{GG20} to show the above theorem and mainly point out the differences.}
The first and second parts of the above theorem will be proved in Subsection 4.1 and Subsection 4.2 respectively.


\subsection{When \texorpdfstring{$|p(t)|\to1$}{}}
 We introduce a reduced system with  
\begin{equation*}
  \eta=|b|^2,\quad\gamma=M(1-|p|^2),\quad \zeta=Mc\overline{bp} , 
\end{equation*}
which satisfy the following ODE system  
\begin{equation}
\label{ODE:bdz11}
\left\{
\begin{aligned}
\eta'+2\nu \eta&=2\text{Im} \zeta,\\
\gamma'&=-2\text{Im} \zeta,\\
\zeta'+( \nu+i(1-\beta)M-i\alpha)\zeta&=i\zeta((3-\beta)\gamma-\eta)-2i\eta\gamma M\\
&\quad+i\gamma^2(M-\gamma+3\eta).
\end{aligned}
\right.
\end{equation}
For $u\in\cW$, notice that $\hat u(k)=cp^{k-1}$, $k\ge 1$, then we have
\begin{align*}
\|u(t)\|_{H^s}^2&=\sum_{k=0}^\infty (1+k^2)^{s}|\hat u(k)|^2\\
&\sim 
\frac{\Gamma(2s+1)M}{(1-|p(t)|^2)^{2s-1}}
=\Gamma (2s+1)M^{2s}\gamma(t)^{1-2s}.
\end{align*}
Hence, we only need to show
\begin{equation} \label{gamma:decay}
    \gamma(t)\sim \frac{\kappa}{t},\quad\kappa=\frac{\nu^2+((1-\beta)M-\alpha)^2}{2\nu M}
\end{equation}
to obtain \eqref{blowup}
\begin{equation*}
        \|u(t)\|_{H^s}^2\sim a^2t^{2s-1}
    \end{equation*}
   with
   \begin{equation*}
       a^2(s,\nu,\alpha,\beta,M)=\Gamma(2s+1)M^{2s}\kappa^{1-2s}.
   \end{equation*}

We observe some facts in this case. The conserved momentum implies that $1-|p|^2$ and $|c|$ decay with the same rate. The integrability and decay of $|b|$ were given in Lemma \ref{limit}. As a consequence, in $(\eta,\gamma,\zeta)$-coordinate one has
\begin{equation}\label{fact0}
    \eta\in L^1(\R_+) \quad \text{and} \quad \zeta={o}(\gamma).
\end{equation}

To show \eqref{gamma:decay}, we take the imaginary part of $\zeta$ equation and use $\gamma'=-2\text{Im} \zeta$ to derive
\begin{equation}\label{equation:zeta}
    \frac{\text{Im}\zeta'}{\nu+i((1-\beta)M-\alpha)}-\frac{\gamma'}{2}=\text{Im}f+\text{Im}r,
\end{equation}
where  
\begin{align*}
    f=i\frac{\gamma^2}{\nu+i((1-\beta)M-\alpha)}(M-\gamma+(3-\beta)\frac{\zeta}{\gamma})
    \end{align*}
    and 
    \begin{align*}
  r=-i\frac{\eta}{\nu+i((1-\beta)M-\alpha)}(\zeta+2\gamma M-3\gamma^2).
\end{align*}
The boundedness of $\eta,\gamma,\zeta$ and the integrability of $\eta$ ensure that $r\in L^1(\R_+)$. As a consequence of \eqref{equation:zeta}, one has  $\text{Im}f\in L^1(\R_+)$. Furthermore, the structure of $f$ ensures that $\gamma^2\in L^1(\R_+)$.

Now, one can integrate the both side of \eqref{equation:zeta} to derive 
\begin{equation*}
 (1+o(1))\gamma(t)=2\int_t^\infty\text{Im}f+2\int_t^\infty\text{Im}r,
\end{equation*}
where in the left-hand side, we use the fact that $\zeta=o(\gamma)$. Computing the right hand side, we have
    \begin{equation*}
    \int_t^\infty \text{Im} f(s)\,ds= \frac{1}{2\kappa}\Big(\int_t^\infty\gamma(s)^2\,ds\Big) (1+o(1))
    \end{equation*}
    and 
    \begin{equation*}
   \int_t^\infty \text{Im}\ r(s)\,ds= O\Big(\int_t^\infty\eta(s)\gamma(s)\,ds\Big)= o(\sup_{s\ge t}\gamma(s)), \end{equation*}
where the last equality holds due to the integrability of $\eta$. 

Now we arrive at
\begin{equation}\label{gamma}
\gamma(t)= \frac{1}{\kappa}\Big(\int_t^\infty\gamma(s)^2\,ds\Big) (1+o(1))+o(\sup_{s\ge t}\gamma(s)).
\end{equation}
We take $\sup_{s\ge t}$ on the above equality to get
\begin{equation*}
  \sup_{s\ge t}\gamma(s)= \frac{1}{\kappa}\Big(\int_t^\infty\gamma(s)^2\,ds\Big) (1+o(1)).
\end{equation*}
Inserting this equality in the equation \eqref{gamma} gives 
\begin{equation*}
  \gamma(t)= \frac{1}{\kappa}\Big(\int_t^\infty\gamma(s)^2\,ds\Big) (1+o(1)).
\end{equation*}
Solving this integral equation gives
\begin{equation*}
 \gamma(t)=\frac{\kappa}{t}(1+o(1))
\end{equation*}
as desired.


\subsection{When  \texorpdfstring{$|p(t)|\to0$}{}}
In this subsection, we investigate the trajectories 
$\{u(t)\}$ with momentum $M$, which do not lie in $\cC_M$ but converge to it. 
As a consequence of Lemma \ref{limit} , the $L^2$ norm of $u(t)$ decays to the momentum $M$, namely 
\begin{align*}
    &\|u(t)\|_{L^2}^2=|b(t)|^2+M(1-|p(t)|^2)\to M,\\
    \text{and} & \quad  \|u(t)\|_{L^2}^2\ge M,\quad \forall t\ge 0.
\end{align*}
We first define the set $\Sigma_{M,\nu,\alpha,\beta}\subset\mathcal{E}_M$ as following
\begin{equation}\label{set:S}
   \Sigma_{M,\nu,\alpha,\beta}=\{u_0\in \mathcal{E}_M\setminus \mathcal{C}_M\mid  \|\lS_{\nu,\alpha,\beta}(t)u_0\|_{L^2}^2\ge M ,\,\forall t\ge 0\}.
\end{equation}
At the end of this subsection, we will see that $\Sigma_{M,\nu,\alpha,\beta}$ is a three-dimensional submanifold in $\mathcal{E}_M$. We first observe some facts of the trajectories with $u_0\in \Sigma_{M,\nu,\alpha,\beta}$:
 \begin{itemize}
    \item Since $ \|u(t)\|_{L^2}^2=|b|^2+M(1-|p|^2)\ge M$, one has
    \begin{equation}\label{fact:1}
      |b(t)|^2\ge M|p(t)|^2.  
    \end{equation}
    \item One has 
    \begin{equation}\label{fact:2}
   b(t)\neq 0,\quad  \forall t\in \R.
    \end{equation}
    Otherwise, $b$ and $p$ would cancel at the same time and the trajectory would not  be  disjoint from $\mathcal C_M$.
    
   \item By Lemma \ref{limit}, the inequality \eqref{fact:1}, and the identity $|c(t)|^2=M(1-|p(t)|^2)^2$, one has   
   \begin{equation}\label{fact:3}
 |b(t)|^2,\,|p(t)|^2\in L^1(\R_+) \quad\text{and}\quad |c(t)|^2\in L^\infty(\R_+)
 \end{equation}

 \end{itemize}

We are going to show the second part of Theorem \ref{twocases}
in the following four steps: In step 1, we show that $u(t)$ converges to $\mathcal{C}_M$ as $t\to+\infty$; In step 2, we establish a scattering property of a reduced system related to $b,c,p$; In step 3, the asymptotic behavior of $u(t)$ can be recovered on the basis of step 2. The geometric structure of $\Sigma_{M,\nu,\alpha,\beta}$ will be discussed in step 4.

\vspace{5mm}
{\bf{Step 1: Convergence of $c(t)$ and $\frac{\overline{p(t)}}{b(t)}$.}}
We show that there exists $\theta\in\mathbb{T}$ such that 
\begin{equation*}
\label{convergence:c}
 e^{itM(1-\beta)}c(t)\to c_\infty=\sqrt M e^{i\theta}   
\end{equation*}
 and 
 \begin{equation}
 \label{convergence:p,b}
 e^{-itM(1-\beta)}\sqrt{M} \frac{\overline{p(t)}}{b(t)}
\to \left(\frac{\varsigma\rho-\alpha+i(\nu-\sigma)}{2M}-1\right)e^{-i\theta} 
 \end{equation}
 as $t\to+\infty$.
Here 
\begin{equation}\label{sign}\varsigma=\text{sgn}\,(\alpha+2M)\in\{-1,1\}
\end{equation} 
and $\sigma,\rho$ are real non negative numbers satisfying
\begin{equation}\label{rhosigma}\sigma^2-\rho^2=\nu^2-\alpha^2-4\alpha M\text{ and } \varsigma\sigma\rho=\nu(\alpha+2M).
\end{equation}
In particular 
 \begin{equation}\label{sigmaa}
      \sigma=\left(\frac{(\nu^2-\alpha^2-4M\alpha)+\sqrt{(\nu^2-\alpha^2-4M\alpha)^2+4\nu^2( \alpha+2 M)^2}}{2}\right)^{\frac12}.
\end{equation}

 We first derive the convergence of $c(t)$. The $c(t)$-equation in the ODE system \eqref{ODE:bcp} implies that 
 \begin{equation}\label{ODE:c}
 i\frac{d}{dt}(e^{itM(1-\beta)}c(t))=e^{itM(1-\beta)}\big[2|b(t)|^2c(t)+2M(1-|p(t)|^2)b(t)p(t)\big].
 \end{equation}
The fact \eqref{fact:3} ensures the integrability of the right-hand side of the above ODE, together with the  conserved momentum, one obtains that
\begin{equation*}
    e^{itM(1-\beta)}c(t)\to c_\infty=\sqrt M e^{i\theta},\quad \theta\in\T.
\end{equation*}
Since $|b(t)|\to0$, we choose $\varepsilon=|b(T)|$ for some $T>>1$. As a consequence of \eqref{fact:1}, one has $|p(T)|\le\frac{\varepsilon}{\sqrt{M}}$.
Furthermore, we claim that 
\begin{equation*}
 |c(T)-c_\infty e^{-iTM(1-\beta)}|= O(\int_T^\infty |b(t)|^2\,dt)=O(\varepsilon^2).
 \end{equation*}
 Indeed, in the above equation, the first equality holds by integration of \eqref{ODE:c}. For the second estimate, one integrates the Lyapunov functional  \eqref{Lyapunov}
 $$\frac{d}{dt}(|b(t)|^2+M(1-|p(t)|^2))=-2\nu |b(t)|^2\ ,$$
from $T$ to $\infty$ to get
 \begin{equation*}
 O(\varepsilon^2)=|b(T)|^2-M|p(T)|^2=2\nu \int_T^\infty |b(t)|^2\,dt.
 \end{equation*}
 It proves the second estimate.

We combine these estimates of $b,c,p$ at time $T$ and the structure of $u$ on $\mathcal{W}$ to obtain the following estimates in any $H^s(\T)$
\begin{equation}\label{dist:2}
 \begin{aligned}
     &\text{dist}\ (u(T), c_\infty e^{-iTM(1-\beta)}e^{ix})=O(\varepsilon) ,\\
     &\text{dist} \ (u(T), b(T)+c_\infty e^{-iTM(1-\beta)}e^{ix}+c_\infty e^{-iTM(1-\beta)}p(T)e^{i2x})=O(\varepsilon^2).
 \end{aligned}    
\end{equation}

Now we are ready to show the convergence of $\sqrt{M}\frac{\overline{p(t)}}{b(t)}$ by a  linearisation argument. Roughly speaking, we use the convergence \eqref{dist:2} to linearise the trajectories after time $T$ and check the $L^2$-norm of the solution. For reader's convenience, we mention that a baby example of the linearisation procedure around solutions with periodic trajectories for the damped Szeg\H{o} equation was provided in \cite{GG20}.

As a consequence of \eqref{dist:2}, we study the following trajectory
\begin{equation*}
     \lS_{\nu,\alpha,\beta}(t) u(T)(x)=e^{-itM(1-\beta)}(c_\infty e^{-iTM(1-\beta)}e^{ix}+\varepsilon v(t,x)+\varepsilon^2w(t,x)),
 \end{equation*}
 where $w$ is uniformly bounded in the sense
 \begin{equation*}
     \|w(t)\|_{H^s}\le C_{s,R},\quad\forall t\in[0,R].
 \end{equation*}
   To derive the equation $v$ satisfied, we calculate the following quantities
\begin{align*}
&e^{itM(1-\beta)}\d_t u^\varepsilon(t)=-iM(1-\beta)(c_\infty e^{-iTM(1-\beta)}e^{ix}+\varepsilon v)+\varepsilon\d_t v+O(\varepsilon^2),\\
 &e^{itM(1-\beta)}( u^\varepsilon(t)\vert\mathds{1})=\varepsilon  (v\vert\mathds{1})+O(\varepsilon^2),\\
&e^{itM(1-\beta)}\Pi(| u^\varepsilon(t)|^2 u^\varepsilon(t))
=Mc_\infty e^{-iTM(1-\beta)} e^{ix}+\varepsilon 2 M  v\\
&\quad\quad\quad\quad\,\quad\quad\quad\quad\quad\quad\quad+\varepsilon c_\infty^2\Pi(e^{-i2TM(1-\beta)}e^{i2x}\overline{v})+O(\varepsilon^2),\\
&e^{itM(1-\beta)}S\Pi(|S^* u^\varepsilon(t)|^2 S^*u^\varepsilon(t))=Mc_\infty e^{-iTM(1-\beta)} e^{ix}+\varepsilon 2Me^{ix}\Pi(e^{-ix}v)\\
&\quad\quad\quad +\varepsilon c_\infty^2 e^{-i2TM(1-\beta)}e^{ix}\Pi(e^{ix}\overline{v})-\varepsilon c_\infty^2 e^{-i2TM(1-\beta)} e^{i2x}\overline{(v\vert\mathds{1})}+O(\varepsilon^2).
\end{align*}
Then $v$ satisfies the equation
 \begin{equation*}
 \begin{aligned}
 i\d_t v + (i\nu-\alpha)(v\vert \mathds{1}) 
 =&c_\infty^2e^{-i2TM(1-\beta)}\Pi(e^{i2x}\overline{v})+ (\beta+1) M  v\\
&-\beta c_\infty^2 e^{-i2TM(1-\beta)}e^{ix}\Pi(e^{ix}\overline{v})-2 \beta Me^{ix}\Pi(e^{-ix}v)\\
&+\beta c_\infty^2 e^{-i2TM(1-\beta)} e^{i2x}\overline{(v\vert\mathds{1})}.
 \end{aligned}
 \end{equation*}
 with the initial value $v(0,x)=\frac{b(T)}{\varepsilon}+c_\infty e^{-iTM(1-\beta)}\frac{p(T)}{\varepsilon}e^{i2x}$.
We observe the equation of $v$ and make the ansatz
\begin{equation*}
    v(t,x)=q_0(t)+q_1(t)e^{ix}+q_2(t)e^{i2x},
\end{equation*}
where
  \begin{align*}
     iq_0'+(i\nu-\alpha)q_0&=(1+\beta)Mq_0+c_\infty^2e^{-2iTM(1-\beta)}\overline q_2,\\
     iq_1'&=(1-\beta) (c_\infty^2e^{-2iTM(1-\beta)}\overline q_1+Mq_1),\\
     iq_2'&=(1-\beta)Mq_2+c_\infty^2e^{-2iTM(1-\beta)}\overline q_0,\\
      q_0(0)=\frac{b(T)}{\varepsilon},\,&q_2(0)=c_\infty e^{-iTM(1-\beta) }\frac{p(T)}{\varepsilon}.
 \end{align*}
We take the derivative of the $q_0$-equation and substitute the equation of $q_2$ to derive
 \begin{align*}
    &q_0''+(\nu +i(\alpha+2\beta M))q_0'-((1-\beta)(i\nu-\alpha)M+\beta^2 M^2)q_0=0,\\
    & q_0(0)=\frac{b(T)}{\varepsilon},\\
    &q_0'(0)= -(\nu+i(\alpha+(1+\beta)M))\frac{b(T)}{\varepsilon}-iMc_\infty e^{-iTM(1-\beta)}\frac{\overline{p(T)}}{\varepsilon}. 
    \end{align*}
 The characteristic equation of this second-order ODE reads 
\begin{equation*}
    \lambda^2+(\nu +i(\alpha+2\beta M))\lambda-((1-\beta)(i\nu-\alpha)M+\beta^2 M^2)=0.
\end{equation*}
The solutions are given by 
 \begin{equation*}
     \lambda_{\pm}=\frac{-(\nu+i(\alpha+2\beta M))\pm(\sigma+i\varsigma\rho )}{2},
 \end{equation*}
 where $\varsigma$, $\rho, \sigma$ are defined by equations \eqref{sign}, \eqref{rhosigma}.

We will prove in Appendix \ref{calculation} that 
$$\frac{\sigma-\nu}{\sigma+\nu}> 0$$ so that $\sigma>\nu$. Hence, the real parts of $\lambda_+$ and $\lambda_-$ admits different signs as
 $$\text{Re}(\lambda_+)=\sigma-\nu>0\quad\text{and}\quad \text{Re}(\lambda_-)=-\sigma-\nu<0.$$
And  hence, the solution $q_0$ is given by
 \begin{equation}\label{q0}
     q_0(t)=A_+e^{\lambda_+t}+A_-e^{\lambda_-t},
 \end{equation}
 where
  \begin{align*}
     A_+(T)&=\frac{q_0'(0)-\lambda_-q_0(0)}{\lambda_+-\lambda_-}\\
     &=-\frac{(\nu+i(\alpha+(1+\beta)M)+\lambda_-)b(T)+iMc_\infty e^{-iTM(1-\beta)}\overline{p(T)}}{\varepsilon(T)(\sigma+i\varsigma\rho)},\\
     A_-(T)&=\frac{\lambda_+q_0(0)-q_0'(0)}{\lambda_+-\lambda_-}\\
    &=\frac{(\nu+i(\alpha+(1+\beta)M)+\lambda_+)b(T)+iMc_\infty e^{-iTM(1-\beta)}\overline{p(T)}}{\varepsilon(T)(\sigma+i\varsigma\rho)}.
 \end{align*}
 
 Now we are ready to check the $L^2$-norm of $\lS_{\nu,\alpha,\beta}(t) u(T)$, especially the two important features: $\|\lS_{\nu,\alpha,\beta}(t) u(T)\|_{L^2}^2\ge M$ and the Lyapunov functional. We fix $T$, the following estimate holds for all $t\in[0,R]$
 \begin{align*}
    0&\le  \|\lS_{\nu,\alpha,\beta}(t) u(T)\|_{L^2}^2-M\\
    &=\|u(T)\|_{L^2}^2-M-2\nu \int_0^t|(\lS_{\nu,\alpha,\beta}(s) u(T)\vert\mathds{1})|^2\,ds\\
     &=|b(T)|^2-M|p(T)|^2-2\nu  \varepsilon^2\int_0^t\big(|(v(s)\vert\mathds{1})|^2+O_{R}(\varepsilon)\big)\,ds.
 \end{align*}
We divide the above inequality both sides by $|b(T)|^2$ to derive
 \begin{align*}
     1-\underbrace{\frac{M|p(T)|^2}{|b(T)|^2}}_{\le 1}-2\nu \int_0^t|\underbrace{(v(s)\vert\mathds{1})}_{=q_0(s)}|^2\,ds\ge -c_R\varepsilon(T),\quad \forall t\in[0,R],
 \end{align*}
 where $c_R$ is a constant depending only on $R$.  
Recall $ q_0(t)=A_+(T)e^{\lambda_+t}+A_-(T)e^{\lambda_-t}$. By computing the integral of $q_0$ and using the above estimate, we infer the existence of a constant $B$ such that 
 \begin{equation*}
       |A_+(T)|^2e^{(\sigma-\nu)R}\le c_R\varepsilon(T)+B,
 \end{equation*}
 We first take upper limit in $T$ of the above inequality
  \begin{equation*}
       \limsup_{T\to+\infty}|A_+(T)|^2e^{(\sigma-\nu)R}\le B.
 \end{equation*}
 Then we take limit in $R\to\infty$, the above inequality holds only if 
\begin{align*}
    \limsup_{T\to+\infty}{|A_+(T)|^2}=0,
\end{align*}
   which implies the convergence \eqref{convergence:p,b}
\begin{align*}
    e^{-iTM(1-\beta)}\sqrt{M} \frac{\overline{p(T)}}{b(T)}
&\to i\frac{\nu+\lambda_-+i(\alpha+(1+\beta)M)}{M}e^{-i\theta}\\
&=(\frac{\varsigma\rho-\alpha+i(\nu-\sigma)}{2M}-1)e^{-i\theta}
\end{align*}
as $T\to+\infty$.

\vspace{5mm}
{\bf{Step 2: Scattering properties for $\eta,\delta,\zeta$ }}
We write 
\begin{equation*}
 \eta=|b|^2,\,\delta=M|p|^2,\,\zeta=Mc\overline{pb},
 \end{equation*}
which satisfy the ODE system 
\begin{equation}
\label{ODE:bdz1}
\left\{
\begin{aligned}
\eta'+2\nu \eta&=2\text{Im} \zeta,\\
\delta'&=2\text{Im} \zeta,\\
 \zeta'+(\nu-i(2M+\alpha))\zeta&=-i((3-\beta)\delta+\eta)\zeta-2i\eta\delta(M-\delta)\\
 &\quad +i(M-\delta)^2(\delta+\eta).
\end{aligned}
\right.
\end{equation}

Due to the decay and boundedness of $(b,c,p)$ in \eqref{fact:3}, we have  $(\eta(t),\delta(t),\zeta(t))\to(0,0,0)$ and $\eta\in L^1(\R_+)$.
The Lyapunov functional in the $(\eta,\delta,\zeta)$-coordinate can be written as 
\begin{equation}\label{LA:ed}
    \eta(t)-\delta(t)=2\nu\int_t^\infty \eta(s)\,ds.
\end{equation}
The  convergence \eqref{convergence:p,b} implies that
\begin{align}
     \frac{\delta(t)}{\eta(t)}
&\to \left|\frac{\varsigma\rho-\alpha+i(\nu-\sigma)}{2M}-1\right|^2,\quad\text{as }t\to +\infty.\label{delta:eta}
\end{align}
We will prove in Appendix \ref{calculation} that 
\begin{equation}\label{EqAppendice}
    \left|\frac{\varsigma\rho-\alpha+i(\nu-\sigma)}{2M}-1\right|^2=\frac{\sigma-\nu}{\sigma+\nu}.
\end{equation}
Combining the Lyapunov functional \eqref{LA:ed} and the convergence \eqref{delta:eta}, the decay rate of $\eta(t)$ is given by
\begin{align}
   \frac{\eta(t)}{\int_t^\infty \eta(s)\,ds} & \to{\sigma+\nu},\quad\text{as }t\to +\infty,\\
\log\big({\int_t^\infty \eta(s)\,ds}\big) & \to -(\sigma+\nu)t(1+o(1)),\quad\text{as }t\to +\infty,\notag 
\end{align}
and 
\begin{equation}
\label{estimate:eta}    
 \eta(t)\le C_\varepsilon  e^{-(\sigma+\nu -\varepsilon)t},\quad t\ge0,\text{ for any }\varepsilon>0.
\end{equation}
Notice that the estimate \eqref{estimate:eta} holds also for $\delta(t)$ and $|\zeta(t)|$.

We write \begin{equation*}
    X=\begin{pmatrix}
    \eta(t)\\
    \delta(t)\\
    \zeta_R(t)=\text{Re}\zeta(t)\\
    \zeta_I(t)=\text{Im}\zeta(t)
    \end{pmatrix}\in\R^4,
\end{equation*}
then the ODE system of $(\eta,\delta,\zeta)$ can be written as
\begin{equation}\label{equation:X1}
    X'+AX=Q(X),
\end{equation}
where 
\begin{equation*}
    A=\begin{pmatrix}
    2\nu &  0  &  0&  -2\\
    0       &  0  &   0&  -2\\
    0       &  0  &   \nu&  2M+\alpha\\
    -M^2    & -M^2&-(2M+\alpha)&  \nu
    \end{pmatrix},
\end{equation*}
and 
\begin{equation*}
    Q(X)=
    \begin{pmatrix}
  0 \\
    0      \\
    (\eta+(3-\beta)\delta)\zeta_I    \\
     -(\eta+(3-\beta)\delta)\zeta_R-2M\delta^2-4M\eta\delta+\delta^3+3\eta\delta^2
    \end{pmatrix}.
\end{equation*}
The matrix $A$ has the eigenvalues
\begin{equation*}
    \nu\pm\sigma,\quad \nu\pm i\rho.
\end{equation*}
Now we write the ODE \eqref{equation:X1} as 
\begin{equation*}
    \frac{d}{dt}(e^{tA}X(t))=e^{tA}Q(X(t)),
\end{equation*}
where the solution $X(t)$ is given by the Duhamel formula
\begin{equation}\label{Duhamel1}
    X(t)=e^{-tA}X_\infty-\int_t^\infty e^{(s-t)A}Q(X(s))\,ds.
\end{equation}
We use estimate \eqref{estimate:eta}. Since $\sigma+\nu$ is the largest eigenvalue of $A$, and $Q$ is a quadratic-cubic form of $X$, one has
\begin{equation*}
\begin{aligned}
|X(t)|&\lesssim_\varepsilon e^{-(\sigma+\nu-\varepsilon)t},\\ 
    |e^{tA}(Q(X(t)))|&\lesssim_\varepsilon e^{-(\nu+\sigma-2\varepsilon)t},\\
    |e^{(s-t)A}(Q(X(s)))|&\lesssim_\varepsilon e^{(\sigma+\nu)(s-t)-2(\sigma+\nu-\varepsilon)s},\quad s,\,t\ge 0 \end{aligned}
\end{equation*}
and
\begin{equation*}
 \int_t^\infty |e^{(s-t)A}(Q(X(s)))|\,ds\lesssim_\varepsilon e^{-2(\sigma+\nu-\varepsilon)t},\quad t\ge 0.    
\end{equation*}
  
We substitute the above estimates to \eqref{Duhamel1} to obtain
\begin{equation}
    e^{-tA}X_\infty=O(e^{-(\sigma+\nu)t}).
\end{equation}

This shows that $X_\infty$ is an eigenvector of $A$ for the eigenvalue $\sigma+\nu$. Consequently, there exists a constant $\eta_\infty \in\R$, such that
$$
    X_\infty=\eta_\infty\left(\begin{array}{c}
      1\\
      \displaystyle\frac{\sigma-\nu}{\sigma+\nu}\\
     \displaystyle (2M+\alpha)\frac{\nu-\sigma}{2\sigma}\\
      \displaystyle\frac{\nu-\sigma}{2}
    \end{array}\right).
$$
We substitute $X_\infty$ to \eqref{Duhamel1} to conclude that there exists $\eta_\infty>0$ (since $\eta_\infty(t)>0$) such that 
\begin{equation}\label{as}
\begin{aligned}
    \eta(t)&=\eta_\infty e^{-(\sigma+\nu)t}(1+O(e^{-(\sigma+\nu)t})),\\
    \delta(t)&=\frac{\sigma-\nu}{\sigma+\nu}\eta_\infty e^{-(\sigma+\nu)t}(1+O(e^{-(\sigma+\nu)t})),\\
    \zeta(t)&=\big(\frac{\varsigma\rho-\alpha+i(\nu-\sigma)}{2}-M\big)\eta_\infty e^{-(\sigma+\nu)t}(1+O(e^{-(\sigma+\nu)t})),
\end{aligned}    
\end{equation}
where the last equality holds since, using \eqref{rhosigma}
$$
(2M+\alpha)\frac{(\nu-\sigma)}{2\sigma}+i\frac{\nu-\sigma}{2}
=\frac{\varsigma\rho-\alpha+i(\nu-\sigma)}{2}-M.$$

Conversely, we claim that 
for every $\eta_\infty>0$, there exists a unique triple $(\eta,\delta,\zeta)$ such that 
\begin{equation}
\left\{
\begin{aligned}
\eta'+2\nu \eta&=2\text{Im} \zeta,\\
\delta'&=2\text{Im} \zeta,\\
 \zeta'+(\nu-i(2M+\alpha))\zeta&=-i((3-\beta)\delta+\eta)\zeta-2i\eta\delta(M-\delta)\\
 &\quad+i(M-\delta)^2(\delta+\eta)
\end{aligned}
\right.
\end{equation}
satisfying the asymptotic behavior \eqref{as}. Indeed, by a fixed point argument one can solve \eqref{Duhamel1} on $[T,\infty)$, where $T$ is large enough such that $|X_\infty|e^{-(\nu+\sigma)T}\lesssim 1$, with the norm 
\begin{equation*}
    \|X\|_T:=\sup_{t\ge T}e^{(\nu+\sigma)t}|X(t)|.
\end{equation*}
 Then the extension to the whole real line is ensured by, say, the identities
$$|\zeta |^2=(M-\delta )^2\eta \delta \ ,\ \delta (t)+2\nu \int_t^\infty \eta (s)\, ds=\eta (t)$$
which, combined with the first equation,  lead to
$$ | \dot \eta |=O(\eta )\ .$$

Furthermore,  following the same argument as for the damped Szeg\H{o} equation in \cite{GG20}, the lower (upper) bounds of initial value $\eta(0)$ ensure the lower (upper) bounds for $\eta_\infty$. Namely, for every $C>0$, there exists $C'>0$ such that
\begin{itemize}
    \item if $\eta(0)\ge C^{-1}$, then $\eta_\infty\ge(C')^{-1}$,
    \item if $\eta(0)\le C$, then $\eta_\infty\le C'$.
\end{itemize}

\vspace{5mm}
{\bf Step 3: Asymptotic behavior for $b,c,p$ }

We first show that there exists 
\begin{equation*}
    (\eta_\infty,\theta,\varphi)\in(0,\infty)\times\T\times\T
\end{equation*}
such that, as $t\to+\infty$
\begin{equation}\label{DAS:scattering}
\begin{aligned}
    b(t)&\sim \sqrt{\eta_\infty}e^{-\frac{\nu+\sigma}{2}t-it(2M+\alpha)(\frac{\nu+\sigma}{2\sigma})+i\varphi},\\
    c(t)&\sim \sqrt{M} e^{-itM(1-\beta)+i\theta},\\
    p(t) & \sim \sqrt{\frac{\eta_\infty}{M}}\left(\frac{\varsigma\rho-\alpha+i(\nu-\sigma)}{2M}-1\right) e^{-\frac{\nu+\sigma}{2}t+it\frac{(\alpha+2M\beta)\sigma+(2M+\alpha)\nu}{2\sigma}+i(\theta-\varphi)}.
\end{aligned} 
\end{equation}
Notice that,  the convergence of $c(t)$ was shown in \eqref{convergence:c}. We combine the equations for $b$ and $\eta$, $p$ and $\delta$ to derive
\begin{align*}
    i\frac{d}{dt}\big(\frac{b}{\sqrt{\eta}}\big)&=\big(\eta-2\delta+2M+\alpha+\frac{\text{Re} \zeta }{\eta}\big)\frac{b}{\sqrt{\eta}}\\
    &=\big((2M+\alpha)\frac{\nu+\sigma}{2\sigma}+O(e^{-(\nu+\sigma)t})\big)\frac{b}{\sqrt{\eta}}\\
     i\frac{d}{dt}\big(\frac{\sqrt{M}p}{\sqrt{\delta}}\big)&=\big((1-\beta)(M-\delta)+\frac{\text{Re} \zeta }{\delta }\big)\frac{\sqrt{M}p}{\sqrt{\delta}}\\
     &=\big(-\frac{(\alpha+2M\beta)\sigma+(2M+\alpha)\nu}{2\sigma}+O(e^{-(\nu+\sigma)t})\big)\frac{\sqrt{M}p}{\sqrt{\delta}}.
\end{align*}
Then there exist $\varphi, \psi$ such that 
\begin{align*}
    b(t) & \sim \sqrt{\eta_\infty}e^{-\frac{\nu+\sigma}{2}t-it(2M+\alpha)\frac{\nu+\sigma}{2\sigma}+i\varphi},\\
    p(t) & \sim \sqrt{\frac{\eta_\infty}{M}}\Big(\frac{\sigma-\nu}{\sigma+\nu}\Big)^{\frac12} e^{-\frac{\nu+\sigma}{2}t+it\frac{(\alpha+2M\beta)\sigma+(2M+\alpha)\nu}{2\sigma}+i\psi}.
\end{align*}
On the other side, we recall the convergence \eqref{convergence:p,b}
\begin{equation*}
 e^{-itM(1-\beta)} \sqrt{M} \frac{\overline{p(t)}}{b(t)}
\to \left(\frac{\varsigma\rho-\alpha+i(\nu-\sigma)}{2M}-1\right)e^{-i\theta}
\end{equation*}
and the equality \eqref{EqAppendice} 
\begin{equation*}
  \big(\frac{\sigma-\nu}{\sigma+\nu}\big)^{\frac12}=\Big|\frac{\varsigma\rho-\alpha+i(\nu-\sigma)}{2M}-1\Big|.  
\end{equation*}
This implies that  
\begin{equation*}
     p(t)  \sim \sqrt{\frac{\eta_\infty}{M}}\left(\frac{\varsigma\rho-\alpha+i(\nu-\sigma)}{2M}-1\right) e^{-\frac{\nu+\sigma}{2}t+it\frac{(\alpha+2M\beta)\sigma+(2M+\alpha)\nu}{2\sigma}+i(\theta-\varphi)}.
\end{equation*}

Now we show that the asymptotic behavior \eqref{DAS:scattering} holds conversely. Namely, for a fixed $(\eta_\infty,\theta,\varphi)\in(0,\infty)\times\T\times\T$, there exists a unique trajectory  \begin{equation*}
u(t,x)=b(t)+\frac{c(t)e^{ix}}{1-p(t) e^{ix}}
\end{equation*}
satisfying the asymptotic behavior \eqref{DAS:scattering}.
From step 2, for fixed $(\eta_\infty,\theta,\varphi)\in (0,\infty)\times\T\times\T$ there exists a unique triplet $(\eta,\delta,\zeta)$ satisfying \eqref{ODE:bdz1} such that \begin{align*}
    \eta(t)&=\eta_\infty e^{-(\sigma+\nu)t}(1+O(e^{-(\sigma+\nu)t})),\\
    \delta(t)&=\frac{\sigma-\nu}{\sigma+\nu}\eta_\infty e^{-(\sigma+\nu)t}(1+O(e^{-(\sigma+\nu)t})),\\
    \zeta(t)&=\left(\frac{\varsigma\rho-\alpha+i(\nu-\sigma)}{2}-M\right)\eta_\infty e^{-(\sigma+\nu)t}(1+O(e^{-(\sigma+\nu)t})).
\end{align*}

Due to the structure of $\eta,\delta,\zeta$,  there exists a fixed large enough $T>0$ such that  $M>\delta(T)>0$, $\zeta(T)\neq 0$ and $\eta(T)>0$. Then there exists $(b_1,c_1,p_1)$ solving the ODE system \eqref{ODE:bcp} such that 
\begin{equation*}
    b_1(T)=\sqrt{\eta(T)},\quad \sqrt{M}p_1(T)=\sqrt{\delta(T)},\quad Mc_1(T)=\frac{\zeta(T)}{\overline{b_1(T)} \overline{p_1(T)}}.
\end{equation*}
Furthermore, due to the uniqueness of the Cauchy problem of the ODE system for $(\eta,\delta,\zeta)$, the above equations hold for all $t\in\R$.
On the other hand, $(b_1,c_1,p_1)$ satisfies the asymptotic behavior \eqref{DAS:scattering} with a pair  $(\theta_1,\varphi_1)\in\T\times\T$, i.e. 
\begin{align*}
     b_1(t) & \sim \sqrt{\eta_\infty}e^{-\frac{\nu+\sigma}{2}t-it(2M+\alpha)\frac{\nu+\sigma}{2\sigma}+i\varphi_1},\\
    c_1(t)&\sim \sqrt{M} e^{-itM(1-\beta)+i\theta_1},\\
    p_1(t) & \sim \sqrt{\frac{\eta_\infty}{M}}\left(\frac{\varsigma\rho-\alpha+i(\nu-\sigma)}{2M}-1\right) e^{-\frac{\nu+\sigma}{2}t+it\frac{(\alpha+2M\beta)\sigma+(2M+\alpha)\nu}{2\sigma}+i(\theta_1-\varphi_1)}.
\end{align*}
Then the triplet $(b,c,p)$ with
\begin{equation*}
    b(t)=e^{i(\varphi-\varphi_1)}b_1(t),\quad c(t)=e^{i(\theta-\theta_1)}c_1(t),\quad p(t)= e^{i(\theta-\varphi-\theta_1+\varphi_1)}p_1(t)
\end{equation*}
satisfies the ODE system \eqref{ODE:bcp} with the desired asymptotic properties. 

At the last step,  we show the uniqueness of the solution $(b,c,p)$.
We assume that $(\tilde b,\tilde c,\tilde p)$
is another solution of the ODE system \eqref{ODE:bcp} with the same asymptotic properties \eqref{DAS:scattering}. Then by the  uniqueness of $(\eta,\delta,\zeta)$, for any $t\in\R$
\begin{equation*}
 |\tilde b(t)|^2=\eta(t),\quad M|\tilde p(t)|^2=\delta(t),\quad M \tilde c(t)\overline{\tilde b(t) } \overline{ \tilde p(t) }=\zeta(t)
\end{equation*}
 and hence, $b-\tilde b$ and $p-\tilde p$ satisfy
\begin{align*}
    i\frac{d}{dt}\big(\frac{b-\tilde b}{\sqrt{\eta}}\big)&=\big((2M+\alpha)\frac{\nu+\sigma}{2\sigma}+O(e^{-(\nu+\sigma)t})\big)\frac{b -\tilde b}{\sqrt{\eta}},\\
     i\frac{d}{dt}\big(\frac{\sqrt{M}(p-\tilde p)}{\sqrt{\delta}}\big)&=\big(-\frac{(\alpha+2M\beta)\sigma+(2M+\alpha)\nu}{2\sigma}+O(e^{-(\nu+\sigma)t})\big)\frac{\sqrt{M}(p-\tilde p)}{\sqrt{\delta}}.
\end{align*}
The first ODE implies 
$$\Big|\frac{b-\tilde b}{\sqrt{\eta}}(t)\Big|\le \int_t^\infty O(e^{-(\nu+\sigma)s}) \Big|\frac{b-\tilde b}{\sqrt{\eta}}(s)\Big|\,ds,$$
Combining this estimate with the boundedness of $|\frac{b-\tilde b}{\sqrt{\eta}}|$, one can show that $b(t)=\tilde b(t)$ for any $t\in\R$ by an iterative argument. By a similar argument, we have $p(t)=\tilde p(t)$ for all $t\in\R$. Since $c(t)\overline{ b(t) } \overline{  p(t) }=\tilde c(t)\overline{\tilde b(t) } \overline{ \tilde p(t) }=\zeta(t)$, we conclude that  $c(t)=\tilde c(t)$, for all $t\in\R$.

\vspace{5mm} 
\textbf{Step 4: Geometric structure of $\Sigma_{M,\nu,\alpha,\beta}\subset\mathcal{E}_M$.}
We define the map
\begin{equation*}
   J:(0,\infty)\times \T\times \T\to \mathcal{E}_M,\quad (\eta_\infty,\theta,\varphi)\mapsto u(0),
\end{equation*}
where $u$ is the unique solution of \eqref{DABS} corresponding to the asymptotic behavior \eqref{DAS:scattering}. Step 3 ensures that the range of $J$ is $\Sigma_{M,\nu,\alpha,\beta}$ and its injectivity is given by the conserved momentum. In order to prove that $\Sigma_{M,\nu,\alpha,\beta}$ is a three-dimensional submanifold of $\mathcal E_M$, one only need to show that $J$ is a proper immersion. 

{\it Smoothness.}
The smoothness of  $J$ with respect to $\eta_\infty$ is a consequence of the fixed point argument in Step 3.
   The dependence with respect to $(\theta,\varphi )$ is much more elementary, since it reflects the gauge and translation invariances, hence it is smooth as well. 

{\it Immersion.}
To prove the immersion property, we just have to check that, for every $(\eta_\infty ,\theta ,\varphi )$, the  three vectors $$\partial_{\eta_\infty}J(\eta_\infty ,\theta ,\varphi ), \partial_\theta J(\eta_\infty ,\theta ,\varphi ), \partial_\varphi J(\eta_\infty ,\theta ,\varphi )$$
are independent. We claim that the subspace spanned by these three vectors is also spanned by $\partial_tu(0), -i\partial_xu(0), iu(0)$. Indeed, in view of Step 3 and of the invariances of equation \eqref{DABS}, one easily checks the following identity,
\begin{equation*}
\begin{aligned}
&e^{i\varphi}\lS_{\alpha,\beta,\nu}(t+T)\left [ J(\eta_\infty, \theta_0, \varphi_0)\right ](x+\theta -\varphi)\\
=&\lS_{\alpha,\beta,\nu}(t)\left[J(\bar{\eta}_\infty,\bar{\theta},\bar{\varphi})\right](x),
\end{aligned}
\end{equation*}
where
\begin{equation*}
\begin{aligned}
&\bar{\eta}_\infty=\eta_\infty{ e}^{-(\sigma+\nu)T},\\
&\bar{\theta}=\theta_0+\theta -M(1-\beta)T,\\
&\bar{\varphi}=\varphi_0+\varphi -(M+\frac{\alpha}{2}) (1+\frac{\nu}{\sigma})T.
\end{aligned}
\end{equation*}
If these three vectors were dependent, this would mean that $u$ is a traveling wave of equation \eqref{DABS}. This would impose that $(u\vert \mathds{1})\equiv 0$, hence $u\in \mathcal C_M$, which is impossible since $\Sigma_{M,\nu,\alpha,\beta}$ is disjoint from $ \mathcal C_M$. 

{\it Closeness.} 
Since $\mathcal C_M$ is compact, it is enough to prove that the closure of
$\Sigma_{M,\nu,\alpha,\beta}$ is contained into  $\Sigma_{M,\nu,\alpha,\beta}\cup \mathcal C_M$. 
Let $(u_n)$ be a sequence of points of $\Sigma_{M,\nu,\alpha,\beta}$ which tends to $u\in \mathcal W$.
Set $(\eta_n,\theta_n,\varphi_n):=J^{-1}(u_n)$. Since $J$ is a homeomorphism onto its range $\Sigma_{M,\nu,\alpha,\beta}$, the only cases to be studied are $\eta_n\to 0$ and $\eta_n\to +\infty $. We use the last part of Step 2 that the lower (upper) bounds of initial value $\eta(0)$ ensure the lower (upper) bounds for $\eta_\infty$. 
As a consequence, in the case $\eta_n\to0$, we obtain that $(u\vert \mathds{1})=0$, and more generally that $(\lS_{\nu,\alpha,\beta} (t)(u)\vert \mathds{1})=0$, so that $u\in \mathcal C_M$. In the second case, we infer $|(u_n\vert \mathds{1})|\to \infty $, which 
contradicts the fact that $u_n$ is convergent.

The proof of Theorem \ref{twocases} is  complete.


\section{Appendix: A calculation}\label{calculation}
In the following, we give the details of the calculations on 
\begin{equation*}
    Z:=\left|\frac{\varsigma\rho-\alpha+i(\nu-\sigma)}{2M}-1\right|^2=\frac{\sigma-\nu}{\sigma+\nu},
\end{equation*}
 where $\varsigma, \sigma ,\rho$ are given by \eqref{sign} and \eqref{rhosigma}.
 We recall the following equalities.
\begin{equation*}\sigma^2-\rho^2=\nu^2-\alpha^2-4\alpha M\text{ and } \varsigma\sigma\rho=\nu(\alpha+2M).
\end{equation*}
 
Remark first that the case $\sigma=\nu$ is excluded. Indeed, if $\sigma=\nu$ then $\rho^2=\alpha^2+4M\alpha$ which is not compatible with the second condition.
We calculate
\begin{align*}
    Z&=\Big|\frac{\varsigma\rho-\alpha+i(\nu-\sigma)}{2M}-1\Big|^2\\
    &=1+\frac{1}{4M^2} \Big((\varsigma\rho-\alpha)^2+(\nu-\sigma)^2-4M(\varsigma\rho-\alpha)\Big)\\
      &=1+\frac{1}{4M^2} \Big(\rho^2+\alpha^2+\nu^2+\sigma^2-2\nu\sigma +4M\alpha-2\varsigma \rho(\alpha+2M)\Big)\\ 
       &=1+\frac{1}{2M^2} 
         (\rho^2+\nu^2)(1-\frac{\sigma}{\nu}).
   \end{align*}
  Then as above 
\begin{align*}
    (1-Z)(\nu+\sigma)
    &=\frac{\rho^2+\nu^2}{2M^2}\frac{\sigma-\nu}{\nu}(\sigma+\nu)\\
    &=\frac{1}{2\nu M^2}(\sigma^2\rho^2+\sigma^2\nu^2-\nu^2\rho^2-\nu^4)\\
     &=\frac{\nu}{2 M^2}((\alpha+2M)^2+\sigma^2-\rho^2-\nu^2)\\
     &=\frac{\nu}{2 M^2}((\alpha+2M)^2-\alpha^2-4M\alpha)\\
     &=2\nu.
\end{align*}
According to the above calculation, one has 
\begin{equation*}
    \frac{2\nu}{1-Z}=\nu+\sigma\quad\text{and} \quad Z=\frac{\sigma-\nu}{\sigma+\nu}.
\end{equation*}

\section*{Acknowledgements}
Z.\ H. thanks for the financial support from Karlsruhe House of Young Scientists (KHYS) through the Research Travel Grant for a stay at the Université Paris-Saclay, which makes the collaboration possible.

\printbibliography
\end{document}